\documentclass[twoside,12pt]{article}

\usepackage{graphicx, subcaption}
\usepackage[top=1in,bottom=1in,left=1in,right=1in]{geometry}
\usepackage[sort&compress]{natbib} 
\usepackage{amsfonts, amsmath, amssymb, amsthm, bbm}
\usepackage{MnSymbol}

\usepackage{mathtools}
\mathtoolsset{showonlyrefs}
\usepackage{enumitem}

\usepackage{xcolor}
\definecolor{darkred}{RGB}{100,0,0}
\definecolor{darkgreen}{RGB}{0,100,0}
\definecolor{darkblue}{RGB}{0,0,150}
\definecolor{ccol}{RGB}{20, 143, 119 }

\usepackage{hyperref}
\hypersetup{colorlinks=true, linkcolor=darkred, citecolor=darkgreen, urlcolor=darkblue}
\usepackage{url}

\usepackage[all]{xy}
\xyoption{rotate}	
\usepackage{tikz-cd}
\usetikzlibrary{automata,positioning,arrows}

\def\H{{\rm H}}

\def\core{{\rm C}}

\def\ball{\mathsf{B}}

\def\d{{\rm d}}
\def\deg{{\rm deg}}

\def\opt{{\rm opt}}

\def\({\left(}
\def\){\right)}

\newtheorem{thm}{Theorem}[section]
\newtheorem{prp}[thm]{Proposition}
\newtheorem{lem}[thm]{Lemma}

\theoremstyle{remark}

\newtheorem{rem}[thm]{Remark}

\def\beq{\begin{equation}} % \setcounter{equation}{1}}
\def\eeq{\end{equation}}
\def\beqn{\begin{eqnarray*}}
\def\eeqn{\end{eqnarray*}}
\def\Bitem{\begin{itemize}\setlength{\itemsep}{.2in}}
\def\bitem{\begin{itemize}\setlength{\itemsep}{.05in}}
\def\eitem{\end{itemize}}
\def\Benum{\begin{enumerate}\setlength{\itemsep}{.2in}}
\def\benum{\begin{enumerate}\setlength{\itemsep}{.05in}}
\def\eenum{\end{enumerate}}
\def\bmult{\begin{multline*}}
\def\emult{\end{multline*}}
\def\bcenter{\begin{center}}
\def\ecenter{\end{center}}
\def\bframe{\begin{frame}}
\def\eframe{\end{frame}}

\newcommand{\thmref}[1]{Theorem~\ref{thm:#1}}
\newcommand{\prpref}[1]{Proposition~\ref{prp:#1}}
\newcommand{\lemref}[1]{Lemma~\ref{lem:#1}}
\newcommand{\secref}[1]{Section~\ref{sec:#1}}
\newcommand{\figref}[1]{Figure~\ref{fig:#1}}

\newcommand{\remref}[1]{Remark~\ref{rem:#1}}

\DeclareMathOperator{\diam}{diam}

\def\cB{\mathcal{B}}

\def\cG{\mathcal{G}}

\def\cS{\mathcal{S}}

\def\cX{\mathcal{X}}

\def\bbN{\mathbb{N}}

\def\bbR{\mathbb{R}}

\newcommand{\inner}[2]{\langle #1, #2 \rangle}

\def\1{\mathbbm{1}}

\def\ve{\varepsilon}

\def\ind{\mathbbm{1}}

\newcommand{\Card}{\operatorname{Card}}

\begin{document}
\thispagestyle{empty}

\title{The Coreness and H-Index \\ of Random Geometric Graphs}

\author{
	Eddie Aamari%
	\footnote{D\'ept. de Math\'ematiques et Applications,
       Ecole Normale Sup\'erieure, Université PSL, CNRS 
---       Paris, France (\url{https://www.math.ens.psl.eu/\~eaamari/})}
	\and
	Ery Arias-Castro%
	\footnote{Dept. of Mathematics \& Halıcıoğlu Data Science Institute, 
       U. of California, San Diego
--- La Jolla, U.S.A. (\url{https://math.ucsd.edu/\~eariasca/})}
	\and
	Cl\'ement Berenfeld%
	\footnote{Institut für Mathematik, U. Potsdam
--- Potsdam, Germany (\url{https://cberenfeld.github.io/})}
}

\date{}
\maketitle

\begin{abstract}
In network analysis, a measure of node centrality provides a scale indicating how central a node is within a network. The \emph{coreness} is a popular notion of centrality that accounts for the maximal smallest degree of a subgraph containing a given node. In this paper, we study the coreness of random geometric graphs and show that, with an increasing number of nodes and properly chosen connectivity radius, the coreness converges to a new object, that we call the \emph{continuum coreness}. In the process, we show that other popular notions of centrality measures, namely the H-index and its iterates, also converge under the same setting to new limiting objects.
\end{abstract}

\maketitle

\section{Introduction} \label{sec:intro}

In network analysis, notions of {\em node centrality} provide a way to quantify how central or important a node is within the network. 
Quite a few notions have been proposed. These include the degree, which is perhaps the most obvious one; the H-index proposed by \citet{hirsch2005index} to ``index to quantify an individual's scientific research output"; and the coreness or shell index proposed for the analysis of social networks \citep{seidman1983network}. Other notions have been proposed, including some based on graph distances \citep{freeman1978centrality} or on (shortest-)path counting~\citep{freeman1977set}; and still others that rely on some spectral properties of the graph \citep{katz1953new, bonacich1972factoring, page1999pagerank, kleinberg1999hubs}.
Notions of centrality are surveyed in~\citep{freeman1978centrality, borgatti2006graph}, and in the reference book~\citep[Ch 4]{kolaczyk2009statistical}.

\subsection{Contribution}

The coreness of a node $i$ in a graph $G$ is defined as the maximal minimal degree of a subgraph $H \subset G$ that contains $i$. That is,
\begin{align*}
\max\left\{ \ell ~\middle|~
\text{there is a subgraph $H$ of $G$ with $i \in H$ and $\min_{j \in H} \deg_H(j) \geq \ell$}
\right\}.
\label{eq:coreness-subgraphs}
\end{align*}
The coreness is a global quantity, in the sense that it depends on the whole geometry of $G$ and not on a small neighborhood of $i$. 
This global sensitivity explains its popularity as a notion of centrality, as it can be seen as  a robustified version of the degree. 
On the other hand, this global nature makes the study of its behavior difficult, especially for random graphs.

This work provides a complete description at first order
of the coreness of random geometric graphs in the large-size limit.
Along the way, we also study the properties of the limiting objects that arise through the study.
Namely, since the asymptotic study of the coreness builds upon its reformulation as iterates of the H-index \citep{lu2016h},
we also provide a full fledged study of the latter.
In a nutshell, we find that the finite iterates of the H-index, converge either to the underlying density in the regime where the connectivity radius goes to zero, or to new objects that we call \emph{continuum H-indices} if the connectivity radius is fixed.
The study of the coreness is substantially more involved, but yields even more interesting findings, as the limiting objects differ from the underlying density even in the regime where the connectivity radius tends to zero.

Our contribution is
in line with the body of work studying asymptotic global properties of random geometric graphs such as, for example, the connection between the discrete and continuous notions of Cheeger constant~\citep{trillos2016consistency, arias2012normalized, muller2020optimal}, 
the domination number \citep{bonato2015}, and other percolation-related properties~\citep{Balister2008}. 

\subsection{Setting}\label{sec:setting}
We consider a multivariate setting where 
\beq\label{setting}
\text{$\cX_n = \{X_1, \dots, X_n\}$ is an i.i.d.~sample from a uniformly continuous density $f$ on $\bbR^d$.} 
\eeq
We denote by $P$ the law of $X_i$ and by $P_n$ the empirical distribution of the sample. The bridge between point clouds and graphs is the construction of a neighborhood graph.  More specifically, for an arbitrary set of distinct points, $x_1, \dots, x_k \in \bbR^d$ and a radius $r > 0$, let $\cG_r(\{x_1, \dots, x_k\})$ denote the graph with node set $V = \{1, \dots, k\}$ and edge set $E = \{(i, j) : \|x_i - x_j\| \le r\}$, where $\|\cdot\|$ denotes the Euclidean norm.  Note that the resulting graph is undirected.
Although it is customary to weigh the edges by the corresponding pairwise Euclidean distances --- meaning that an edge $(i,j)$ has weight $\|x_i -x_j\|$ --- we choose to focus on the purely combinatorial degree-based properties of the graph, so that it is sufficient to work with the unweighted graph.
When the graph is built on a random sample of points, it is sometimes called a {\em random geometric graph} \citep{penrose2003random}. The connectivity radius may depend on the sample, although this dependency will be left implicit for the most part.

Everywhere, $\ball(x,r)$ will denote the closed ball centered at $x$ and of radius $r$.
For a measurable set $A$, $|A|$ will denote its volume. In particular, we will let $\omega$ denote the volume of the unit ball, so that $|\ball(x, r)| = \omega r^d$ for all $x \in \bbR^d$ and $r\geq 0$.
We will let 
\beq\label{N}
N := n \omega r^d,
\eeq
which as we shall see, will arise multiple times as a normalization factor.

\subsection{Outline}

With the notation just introduced, we study the large-sample ($n\to\infty$) limit of the centrality of $x$ in the random neighborhood graph $\cG_r(x, \cX_n) := \cG_r( \{x\} \cup \cX_n)$, where the sample $\cX_n$ is generated as in \eqref{setting}.
More specifically, we focus on the degree $\deg_r(x, \cX_n)$; on the $k$th iterate of the H-index $\H_r^k(x, \cX_n)$; and on the coreness $\core_r(x, \cX_n)$.
See \figref{diagram_recap} for a compact summary of the main results that we derive in the paper.

\begin{figure}[h!]
\centering
        \begin{tikzpicture}[on grid, sloped] %node distance=4cm,
            \node[draw=none] (hkrn) {$\frac1{N} \H^k_r(x,\cX_n)$};
            \node[draw=none] (crn) [below=10em of hkrn] {$\frac1{N} \core_r(x,\cX_n)$};
            \node[draw=none] (hkr) [right=12em of hkrn] {$\H^k_r f_r(x)$};
            \node[draw=none] (f) [above right =6em and 12em of hkr] {$f(x)$};
            \node[draw=none] (cr) [below=10em of hkr] {$\core_r(x,f)$};
            \node[draw=none] (c0) [below =10em of f] {$\core_0(x,f)$};
            \path[->]
            (hkrn)    edge  node[above] {\small (\citeauthor[Thm 1]{lu2016h})} node[below] {\small $k \to \infty$} (crn)
            (hkrn)    edge  node[above] {\small \thmref{Hfixedr}} node[below] {$n \to \infty$} (hkr)
            (hkrn)    edge[bend left=15]  node[above] {\small \thmref{H}} node[below] {\small $n \to \infty, r \to 0$} (f)
            (crn)    edge  node[above] {\small \thmref{fixedr}} node[below] {\small $n \to \infty$} (cr)
            (crn)    edge[bend right=55]  node[above] {\small \thmref{main}} node[below] {\small $n \to \infty, r \to 0$} (c0)
            (hkr)    edge  node[above] {\small \prpref{hr-to-f}} node[below] {\small $r \to 0$} (f)
            (hkr)    edge  node[above] {\small \prpref{cr}} node[below] {\small $k \to \infty$} (cr)
            (cr)    edge  node[above] {\small \prpref{c0}} node[below] {\small $r \to 0$} (c0);
        \end{tikzpicture}
\caption{These are the main relationships that we establish. In this diagram, $\H^k_r f_r(x)$ is defined in~\eqref{def:hr} and arises as the large-$n$ limit of $\frac1{N} \H^k_r(x,\cX_n)$; $\core_r(x,f)$ is defined as the large-$k$ limit of $\H^k_r(x,f)$, and is shown to be the large-$n$ limit of $\frac1{N} \core_r(x,\cX_n)$; and $\core_0(x,f)$ is defined as the small-$r$ limit of $\core_r(x,f)$.}
\label{fig:diagram_recap}
\end{figure}

The remaining of the paper is organized as follows.
\secref{h-index} studies the H-index and other finite iterates of the H-index. \secref{core} focuses on the coreness, and is the main part of the paper, being the most technical and also the most interesting, in that the final ($r \to 0$) limiting object is a new function. 
In \secref{numeric} we report on some numerical simulations meant to illustrate the theory and to probe other technical questions that are currently beyond our reach.

\subsection{A preliminary result of stochastic convergence}

The analysis of the centrality measures appearing in this paper relies on the following elementary lemma, which will be used throughout to control stochastic terms. It involves the Vapnik-Chervonenkis dimension of classes of subsets. 
Namely, given a class $\mathcal{S}$ of subsets of $\bbR^d$ and an integer $m \geq 1$, the \emph{scattering coefficient} of $\mathcal{S}$ for $m$ point is defined as
\begin{align*}
\Delta_{\mathcal{S}}(m)
&:=
\max_{x_1,\dots,x_m\in \bbR^d}
\#\{(\ind_{x_1 \in S},\dots,\ind_{x_m\in S}),\: S \in \mathcal{S}\},
\end{align*}
which is the maximum number of different labelings of $m$ points that $\mathcal{S}$ can produce. 
The \emph{Vapnik-Chervonenkis (VC) dimension} of $\mathcal{S}$ is then defined as the maximum number of points that can be arbitrary labeled with $\mathcal{S}$, that is,
\begin{align*}
\mathrm{VC}(\mathcal{S})
:=
\sup\{m\geq 1 , \Delta_{\mathcal{S}}(m) =2^m\}
\in \mathbb{N} \cup \{\infty\}
.
\end{align*}

\begin{lem} \label{lem:eta} Let $(\cS_r)_{r>0}$ be a family of classes of subsets of $\bbR^d$  such that:
\bitem
\item[(i)] The VC-dimension of $\cS_r $ is bounded from above by some $v \in \bbN$ uniformly for all $r > 0$;
\item[(ii)] For all $r > 0$ and $S \in \cS_r$, we have $\diam(S) \leq 2r$.
\eitem
Then, for any sequence $r = r_n$ such that $n r^d \gg \log n$, we have
$$
\eta := \sup_{S \in \cS_r} \frac{1}{\omega r^d} |P_n(S) - P(S)| \xrightarrow[n \to \infty]{} 0~~~\text{a.s.}
$$
where we recall that $P_n$ is the empirical distribution of a $n$-sample drawn from $P$.
\end{lem}

\begin{proof} For $\kappa > 0$ to be chosen later, write
$$
\delta_n := \frac{4}{(2n+1)^{\kappa v}} \quad\text{and}\quad \ve_n := 2 \sqrt{(1+\kappa) v \frac{\log(2n+1)}{n}}.
$$
Using \cite[Thm 5.1]{boucheron2005theory}, we find that the event on which
$$
\sup_{S \in \cS_r} \frac{P(S)-P_n(S)}{\sqrt{P(S)}} \vee \frac{P_n(S)-P(S)}{\sqrt{P_n(S)}} \leq \ve_n
$$
has probability at least $1-2\delta_n$. Using the fact that $P(S) \leq \|f\|_{\infty} \omega (2r)^d$, we find that on this event, there holds
$$
\frac{1}{\omega r^d}(P(S)-P_n(S)) 
%\leq 
%\sqrt{2^d\|f\|_{\infty}}\frac{P(S)-P_n(S)}{\sqrt{\omega r^d} \sqrt{P(S)}} 
\leq 
\sqrt{2^d \|f\|_{\infty}/\omega} \frac{\ve_n}{r^{d/2}}.
$$
On the other hand, $P_n(S)-P(S) \leq \ve_n \sqrt{P(S)}$ yields
$$
\sqrt{P_n(S)} \leq \frac12 \(\sqrt{\ve_n^2 +4 P(S)}+\ve_n\)
,
$$
which in turn implies 
$$
P_n(S) - P(S) \leq \frac12\left(\ve_n^2 + \ve_n \sqrt{\ve_n^2+4P(S)}\right) \leq \frac12\left(\ve_n^2 + \ve_n \sqrt{\ve_n^2+2^{d+2}\|f\|_{\infty}\omega r^d}\right). 
$$
Dividing both sides by $\omega r^d$ yields
$$
\frac{1}{\omega r^d}(P_n(S) - P(S)) \leq \frac{1}{2\omega}\left(\frac{\ve_n^2}{r^d} + \frac{\ve_n}{r^{d/2}} \sqrt{\frac{\ve_n^2}{r^d}+2^{d+2}\|f\|_{\infty}\omega}\right).
$$
All in all, we have proved that
$$
\sup_{S \in \cS_r} \frac{1}{\omega r^d} |P_n(S) - P(S)| = O(\ve_n/r^{d/2})
,
$$
which by assumption goes to $0$ as $n \to \infty$. Taking $\kappa$ large enough so that $\delta_n$ is summable, the Borel-Cantelli lemma then concludes the proof.  
\end{proof}
Note that in particular the result holds if $r = r_0$ is constant and if $\cS_{r_0}$ has finite VC-dimension.

\section{H-Index}
\label{sec:h-index}

The {\em H-index} is named after \cite{hirsch2005index}, who introduced this centrality measure in the context of citation networks of scientific publications.
For a given node in a graph, it is defined as the maximum integer $h$ such that the node has at least $h$ neighbors with degree at least $h$. That is, in our context, the H-index of $x$ in $\cG_r(x, \cX_n)$ writes as
\begin{align*}
\H_r(x, \cX_n) 
:=
\text{largest $h$ such that } \# \big\{X_i \in \ball(x, r):  \deg_r(X_i, \cX_n) \ge h\big\} \ge h,
\end{align*}
where we recall that the degree of $x \in \bbR^d\setminus \cX_n$ in the graph $\cG_r(x, \cX_n)$ is given by
\begin{equation}
\label{eq:degree}
\deg_r(x, \cX_n) := \sum_{i=1}^n \ind_{\|x - X_i\| \le r}.
\end{equation}
Note that in our context, the point set $\cX_n$ is an i.i.d sample with common density $f$ on $\bbR^d$, so that it is composed of $n$ distinct points almost surely. Furthermore, if $x=X_{i_0} \in \cX_n$, the degree of $x$ in the graph $\cG_r(x, \cX_n)$ writes as $\sum_{i \neq i_0} \ind_{\|x - X_i\| \le r} = (\sum_{i=1}^n \ind_{\|x - X_i\| \le r})-1$, and therefore only differs by $1$ from the formula of \eqref{eq:degree}. As this difference will be negligible after renormalization by $1/N$, we will only consider the sum of indicators of \eqref{eq:degree} for simplicity.

The H-index was put forth as an improvement on the total number of citations as a measure of productivity, which in a citation graph corresponds to the degree.
We show below that in the latent random geometric graph model of \eqref{setting}, the H-index can be asymptotically equivalent to the degree --- see Theorems~\ref{thm:Hfixedr} and~\ref{thm:H}. 

\citet{lu2016h} consider iterates of the mechanism that defines the H-indices as a function of the degrees: The second iterate at a given node is the maximum $h$ such that the node has at least $h$ neighbors with H-index at least $h$, and so on.
More generally, given any (possibly random) bounded measurable function $\phi : \bbR^d \to \bbR$, we define the (random) bounded measurable function $\H_{n,r}\phi : \bbR^d \to \bbR$ as 
\begin{align}
\H_{n,r}\phi(x) 
&:= 
\text{largest $h$ such that } \#\big\{X_i \in \ball(x,r): \phi(X_i) \geq h\big\} \geq h
\nonumber
\\
&=
N
\max\left\{
h ~\middle|~
\frac1{N} \sum_{i=1}^n \ind_{\|x-X_i\|\leq r} \ind_{\phi(X_i)/N \ge h} \ge h \right\}
.
\label{def:hnr}
\end{align}
The H-index $\H_r(x,\cX_n)$ can be simply written $\H_{n,r} \deg_{r}(x,\cX_n)$,
and the successive iterations of the H-index $\H^k_r(x,\cX_n)$ are simply $\H^k_{n,r} \deg_{r}(x,\cX_n)$.
Given this reformulation of the iterated H-indices, it is essential to investigate the behavior of  $\H^0_r(x,\cX_n) = \deg_{r}(x,\cX_n)$ first.
We recall this very elementary result regarding the asymptotic properties of the degree in random geometric graphs.

\begin{thm}
\label{thm:degree_rfixed}
(i) If $r>0$ is fixed, then almost surely,
\begin{equation*}
\frac1{N}\, \deg_r(x, \cX_n)
\xrightarrow[n \to \infty]{}
f_r(x)~~\text{uniformly in $x \in \bbR^d$,}
\end{equation*}
where $f_r$ is the convoluted version of the density $f$, namely
\beq
\label{eq:f_r}
f_r(x) := \frac1{|\ball(x,r)|} \int_{\ball(x, r)} f(z) \d z.
\eeq
(ii) If $r = r_n$ is such that $r \to 0$ and $n r^d \gg \log n$, then almost surely,
\begin{equation*}
\frac1{N}\, \deg_r(x, \cX_n) \xrightarrow[n \to \infty]{} f(x) ~~\text{uniformly in $x \in \bbR^d$.}
\end{equation*}  
\end{thm}

\begin{proof}
This comes from a simple application of \lemref{eta} to the collection of sets $\{\cS_r\}_{r > 0}$ with $\cS_r := \{\ball(x,r)~|~x \in \bbR^d\}$, and of the fact that $f_r$ converges uniformly to $f$ since $f$ is assumed to be uniformly continuous on $\bbR^d$.   
\end{proof}

Given the variational formula \eqref{def:hnr}, a natural continuous equivalent of the H-index is the $\H_r$ transform of the density $f$, where $\H_r$ is defined for any non-negative bounded measurable function $\phi : \bbR^d \to \bbR$ as 
\begin{align}
\H_r \phi(x) = \sup\left\{t \geq 0~\middle|~ \frac{1}{\omega r^d} \int_{\ball(x,r)} \ind_{\phi(z) \geq t} f(z) \d z \geq t\right\}.
\label{def:hr}
\end{align}
See \figref{hrphi} for an illustration of this transform. The $k$-th iteration of $\H_r$ applied to $\phi$ is simply denoted by $\H_r^k \phi$. The $\H_r$ transform enjoys a few elementary properties, such as monotonicity, Lipschitzness and modulus of continuity preservation, which we now present. 
The modulus of continuity of a function $g : \bbR^d \to \bbR$ is defined by 
$$
\omega_g(u) := \sup\left\{ |g(x)-g(y)| : \|x-y\| \leq u \right\},$$
for all $u \geq 0$. As in our framework (see \eqref{setting}), $f$ is assumed to be uniformly continuous, $\lim_{u \to 0} \omega_f(u) = 0$. In what follows, $\ell^\infty(\bbR^d)$ denotes the class of bounded measurable maps $\phi : \bbR^d \to \bbR$. 

\begin{figure} 
\centering
\includegraphics[scale = .7]{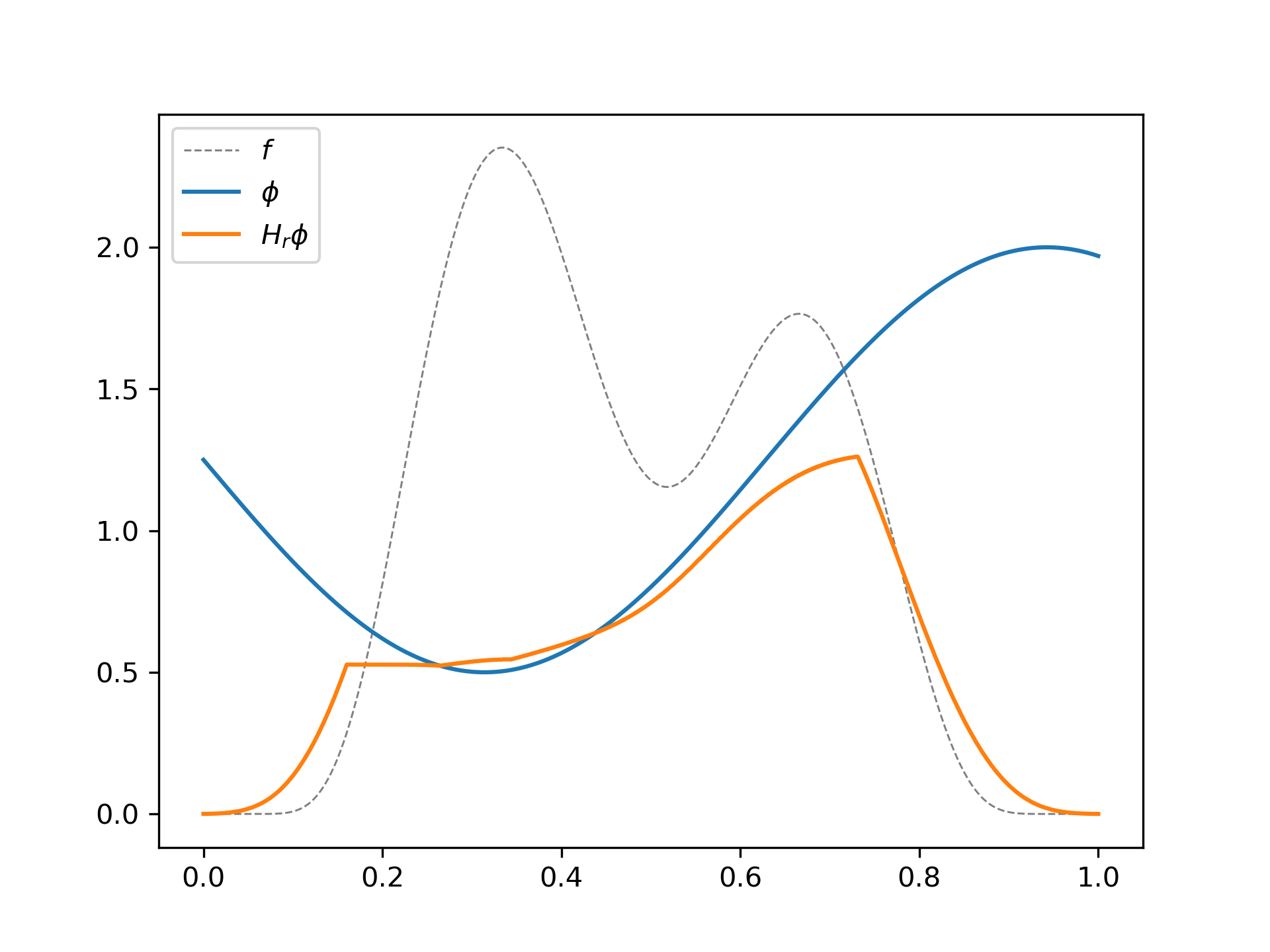}
\caption{
A density $f$, a function $\phi$, and its transform $\H_r \phi$ for $r = 0.1$.
Both $f$ and $\phi$ are smooth. $\H_r \phi$ does not appear to be continuously differentiable everywhere but is nonetheless Lipschitz, with Lipschitz constant no bigger than that of $f$ and $\phi$ (see \lemref{monotonous}).
}
\label{fig:hrphi}
\end{figure}

\begin{lem} \label{lem:monotonous} For all $r > 0$ and $\phi,\psi \in \ell^\infty(\bbR^d)$, there holds:
\benum
\item[1.] (Monotonicity) If $\phi \leq \psi$, then $\H_r \phi \leq \H_r\psi$;
\item[2.] (Lipschitzness) $\|\H_r \phi- \H_r\psi\|_\infty \leq \|\phi-\psi\|_\infty$;
\item[3.] (Regularization) $\omega_{\H_r \phi} \leq \omega_\phi \wedge \omega_f$. In particular, since $\omega_{f_r} \leq \omega_f$, we have $\omega_{\H_r^k f_r} \leq \omega_f$ $\forall k \geq 1$.
\eenum
\end{lem}
\begin{proof}
Point 1.  is trivial once noted that the functional
$$
\phi \mapsto \frac1{\omega r^d} \int_{\ball(x,r)} \ind_{\phi(z) \geq t}f(z) \d z
$$
that appears in the definition of $\H_r$ is non-decreasing in $\phi$.  

For Point 2., let $\ve =  \|\phi-\psi\|_\infty$. We have 
$$
\frac1{\omega r^d} \int_{\ball(x,r)} \ind_{\phi(z) \geq t}f(z) \d z \leq \frac1{\omega r^d} \int_{\ball(x,r)} \ind_{\psi(z) \geq t-\ve}f(z) \d z
$$
so that $\H_r \phi(x) \leq \H_r \psi(x) + \ve$, and the proof follows. 

Finally, for Point 3., Let $x,y \in \bbR^d$, and denote $u = y-x$ and $\ve =\omega_f \vee \omega_\phi(\|x-y\|)$. We have
$$
\frac1{\omega r^d} \int_{\ball(x,r)} \ind_{\phi(z) \geq t}f(z) \d z = \frac1{\omega r^d} \int_{\ball(y,r)} \ind_{\phi(z+u) \geq t}f(z+u) \d z 
\leq 
\frac1{\omega r^d}
\int_{\ball(y,r)} \ind_{\phi(z) \geq t - \ve }f(z) \d z + \ve
$$
so that we immediately find that $\H_r \phi(x) \leq \H_r \phi(y) + \ve$. 
\end{proof}

\subsection{Continuum H-indices: $r > 0$ fixed}
As intuited above, we have the following general convergence result of the random discrete transform $\H_{n,r}$ towards the continuum one $\H_r$.

\begin{lem} \label{lem:hnr} Let $\phi_n, \phi \in \ell^\infty(\bbR^d)$ be random variables such that almost surely, $\frac1N\phi_n \xrightarrow[n \to \infty]{} \phi$ uniformly. Then almost surely, $\frac1N \H_{n,r} \phi_n \xrightarrow[n \to \infty]{} \H_r \phi$ uniformly.
\end{lem}

\begin{proof} Notice that
\begin{align*} 
\H_{n,r}\phi_n(x) &= \sup\{h \geq 0~|~\Card\{X_i \in \ball(x,r), \phi(X_i) \geq h\} \geq h\} \\
&= N \times \sup\{t \geq 0~|~\Card\{X_i \in \ball(x,r), \phi(X_i) \geq Nt\} \geq Nt\} \\
&= N \times  \sup\left\{t \geq 0~\middle|~\frac1N \sum_{i=1}^n \ind_{X_i \in \ball(x,r)} \ind_{\frac1N \phi_n(X_i) \geq t} \geq t\right\}.
\end{align*} 
Let $\ve = \|\frac1N \phi_n - \phi\|_\infty $. We have
$$
\frac1N \sum_{i=1}^n \ind_{X_i \in \ball(x,r)} \ind_{\frac1N \phi_n(X_i) \geq t} \leq \frac1N \sum_{i=1}^n \ind_{X_i \in \ball(x,r)} \ind_{\phi(X_i) \geq t-\ve}.
$$
Note that the class of balls of $\bbR^d$ is a VC-class, and so is the set of super-level sets of $\phi$. As a result, the class 
$$
\cS_r = \{\ball(y,r) \cap \{\phi \geq s\}, y \in \bbR^d, s \geq 0\}
$$
thus satisfies the assumptions of \lemref{eta}. Furthermore, using notation $\eta$ from \lemref{eta}, we get
$$
\frac1N \sum_{i=1}^n \ind_{X_i \in \ball(x,r)} \ind_{\phi(X_i) \geq t-\ve} \leq \frac{1}{\omega r^d} \int_{\ball(x,r)} \ind_{\phi(z) \geq t - \ve} f(z) \d z + \eta
$$
uniformly in $x$ and $t$. We thus have
$$
\frac1N \H_{n,r}\phi_n(x) \leq \sup\left\{t \geq 0~ \middle | ~\frac{1}{\omega r^d} \int_{\ball(x,r)} \ind_{\phi(z) \geq t - \ve} f(z) \d z \geq t - \eta \right\},
$$
yielding $\frac1N \H_{n,r}\phi_n(x) \leq \H_r \phi(x) + \ve \vee \eta$. 
The lower bound can be obtained in the same fashion. We conclude by letting $n \to \infty$, so that $\eta$ goes to $0$ a.s. (\lemref{eta}) and $\ve$ as well by assumption. 
\end{proof}

When applied iteratively to the sequence of degree functions of $\cG_r(x,\cX_n)$, \lemref{hnr} yields the following result.

\begin{thm} \label{thm:Hfixedr} 
If $r > 0$ and $k \in \bbN^*$ are fixed, then almost surely,
$$
\frac1{N}\, \H^k_r(x, \cX_n) \xrightarrow[n \to \infty]{} \H^k_r f_r(x)
~~\text{uniformly in $x \in \bbR^d$.}
 $$
\end{thm}

\begin{proof}
Apply \lemref{hnr} recursively to find that $\frac1N \H^k_{n,r} \phi_n \to \H_r^k \phi$ for all $k \geq 1$. The stated result follows readily starting from $\phi_n = \deg_{r}(\cdot,\cX_n)$ and $\phi = f_r$.
\end{proof}

We see that the iterated continuum $\H$-indices $\H^k_r f_r$ behave very differently from the convoluted likelihood $f_r$, as shown in \figref{hrfr}. 
In particular, for $k\geq 1$, $\H^k_r f_r(x)$ depends on $f$ in an even less local way than $f_r$, since it depends on the values of $f$ on $\ball(x,(k+1)r)$.

\subsection{Continuum H-indices: $r\to 0$}
To gain insights on what the discrete H-indices converge to as $r = r_n \to 0$, let us  first examine how their fixed-$r$ continuous counterparts $\H_r^k f_r$ behave in the same regime.

\begin{prp}
\label{prp:hr-to-f}
For all $k \geq 1$, $\H^k_r f_r(x) \xrightarrow[r \to 0]{} f(x)$ uniformly in $x \in \bbR^d$.
\end{prp}
\begin{proof}
On one hand, we have $\H^k_r f_r(x) \leq f_r(x) \leq f(x) + \omega_f(r)$. 
On the other hand, we get from the definition of $\H_r f_r$ that $\H_r f_r \geq f - \omega_f(r)$. Using this bound recursively together with \lemref{monotonous}, we find that $\H^k_r f_r \geq f - k \omega_f(r)$.
At the end of the day, we have proven that $\|\H_r^k f_r - f\|_\infty \leq k \omega_f(r)$, which concludes the proof. 
\end{proof}

Coming back to the discrete H-indices, we naturally get that the $k$-th iteration of the $\H$-index converges to $f(x)$ as $r =r_n$ converges to $0$ slowly enough.
\begin{thm} \label{thm:H} 
If $r = r_n$ is such that $r \to 0$ and $n r^d \gg \log n$, then for all $k \in \bbN$, almost surely,
\[
\frac1{N}\, \H^k_r(x, \cX_n) \xrightarrow[n \to \infty]{} f(x)
~~\text{uniformly in $x \in \bbR^d$.}
\]
\end{thm}

\begin{proof} 
First, decompose
\[
\left|\frac1N\H^k_r(x,\cX_n) - f(x)\right| \leq \left|  \frac1N\H^k_r(x,\cX_n) - \H_r^k f_r (x)\right| + |\H_r^k f_r (x)-f(x)|.
\]
\prpref{hr-to-f} asserts that the second (deterministic) term converges uniformly to zero as $r \to 0$.
For the first (stochastic) one, we use expressions \eqref{def:hnr} and \eqref{def:hr} of $\H_{n,r}$ and $\H_r$ respectively, and the proof of \thmref{Hfixedr}, to get that
\begin{align*} 
\left|\frac1N \H^k_r(x,\cX_n) - \H_r^k f_r (x)\right|
\leq
\eta := \sup_{S \in \cS_r} \frac{1}{\omega r^d} |P_n(S) - P(S)|,
\end{align*} 
where $P_n(\d z) = n^{-1} \sum_{i = 1}^n \delta_{X_i}(\d z)$, $P(\d z) = f(z) \d z$, and 
\begin{align*}
\cS_r &= \Big\{\ball(y,r) \cap \{\phi \geq s\} ~|~y \in \bbR^d, s \geq 0, \phi \in \{f_r,\dots,\H_r^kf_r\}\Big\}.
\end{align*}
As an intersection class of two VC classes, $\cS_r$ is also VC, with dimension uniformly bounded in $r$. It is composed of sets of radii at most $r$, so that \lemref{eta} applies and yields $\eta \to 0$ almost surely as $n \to \infty$.
\end{proof}

\section{Coreness}
\label{sec:core}
The notion of {\em coreness} is based on the concept of {\em core} as introduced by \cite{seidman1983network}. (Seidman does not mention `coreness' and only introduces cores, and we are uncertain as to the origin of the coreness.)
For an integer $\ell \geq 0$, an $\ell$-core of a given graph is a maximal induced subgraph which has minimum degree $\ell$.  
To be sure, this means that any node in an $\ell$-core is neighbor to at least $\ell$ nodes in that core.
In a given graph, the coreness of a node is the largest integer $\ell$ such that the node belongs to an $\ell$-core.
For a recent paper focusing on the computation of the $\ell$-cores, see~\citep{malliaros2020core}.

The coreness is closely related to the degree and H-index. In fact, \cite[Thm 1]{lu2016h} shows that it arises when iterating the definition of the H-index ad infinitum, when starting with the degree function.
That is, in our context, we will study the random coreness
\begin{align}
\core_r(x,\cX_n)
:=
\H_r^\infty(x,\cX_n)
.
\label{eq:coreness-hindex}
\end{align}
In particular, the coreness satisfies the following fixed-point property: The coreness of node $i$ is the maximum $\ell$ such that at least $\ell$ of its neighbors have coreness at least $\ell$. Said otherwise, it is the maximal minimal degree of a subgraph $H$ that contains $x$:
\begin{align}
\core_r(x,\cX_n)
=
\max\left\{ \ell ~\middle|~
\text{there is a subgraph $H$ of $\cG_r(x,\cX_n)$ with $x \in H$ and $\min_{i \in H} \deg_H(i) \geq \ell$}
\right\}.
\label{eq:coreness-subgraphs}
\end{align}
The coreness was analyzed in the context of an Erd\"os--R\'enyi--Gilbert random graph in a number of papers, for example, in \citep{luczak1991size, janson2008asymptotic, pittel1996sudden, riordan2008k, janson2007simple}, and also in the context of other graph models, for example, in \citep{frieze2009line}.  We are not aware of any work that analyzes the coreness in the context of a random geometric graph.

\begin{rem}
As the non-negative integer sequence $(\H_r^k(x,\cX_n))_{k\geq 0}$ is non-increasing, it becomes stationary after some index $k^\infty < \infty$. Said otherwise, the naive algorithm computing $\H_r^\infty(x,\cX_n)$ by iterating the H-index terminates after a finite number of iterations, so that bounding $k^\infty$ is of particular computational interest. 
Such a bound, depending on the geometric structure of the graph, is discussed in \secref{k_max}.
\end{rem}

\subsection{Continuum coreness: $r>0$ fixed}
As defined above in \eqref{eq:coreness-hindex}, the discrete coreness is obtained by applying the H-index operator to the degree infinitely many times. Having in mind \thmref{Hfixedr}, we naturally define the notion of continuum $r$-coreness by taking the limit of the iterated continuum H-index $\H^k_r f_r(x)$ as the number of iteration $k$ goes to $\infty$.

\begin{prp}
\label{prp:cr}
$\H^k_r f_r (x)$ converges uniformly in $x$ as $k\to \infty$. Its limit, denoted by $\core_r(x,f)$, is called the \emph{continuum $r$-coreness} at $x$.
\end{prp}

\begin{rem}
\label{rem:cr-equicontinuous}
Note that since the convergence is uniform, $\core_r(\cdot, f)$ is uniformly continuous and its modulus of continuity is bounded from above by $\omega_f$ (\lemref{monotonous}). See \figref{hrfr} for an illustration of the convergence of the iterations $\H^k_r f_r$ towards $\core_r(\cdot,f)$.
\end{rem}

\begin{figure}[h!] 
\centering
\includegraphics[scale = .7]{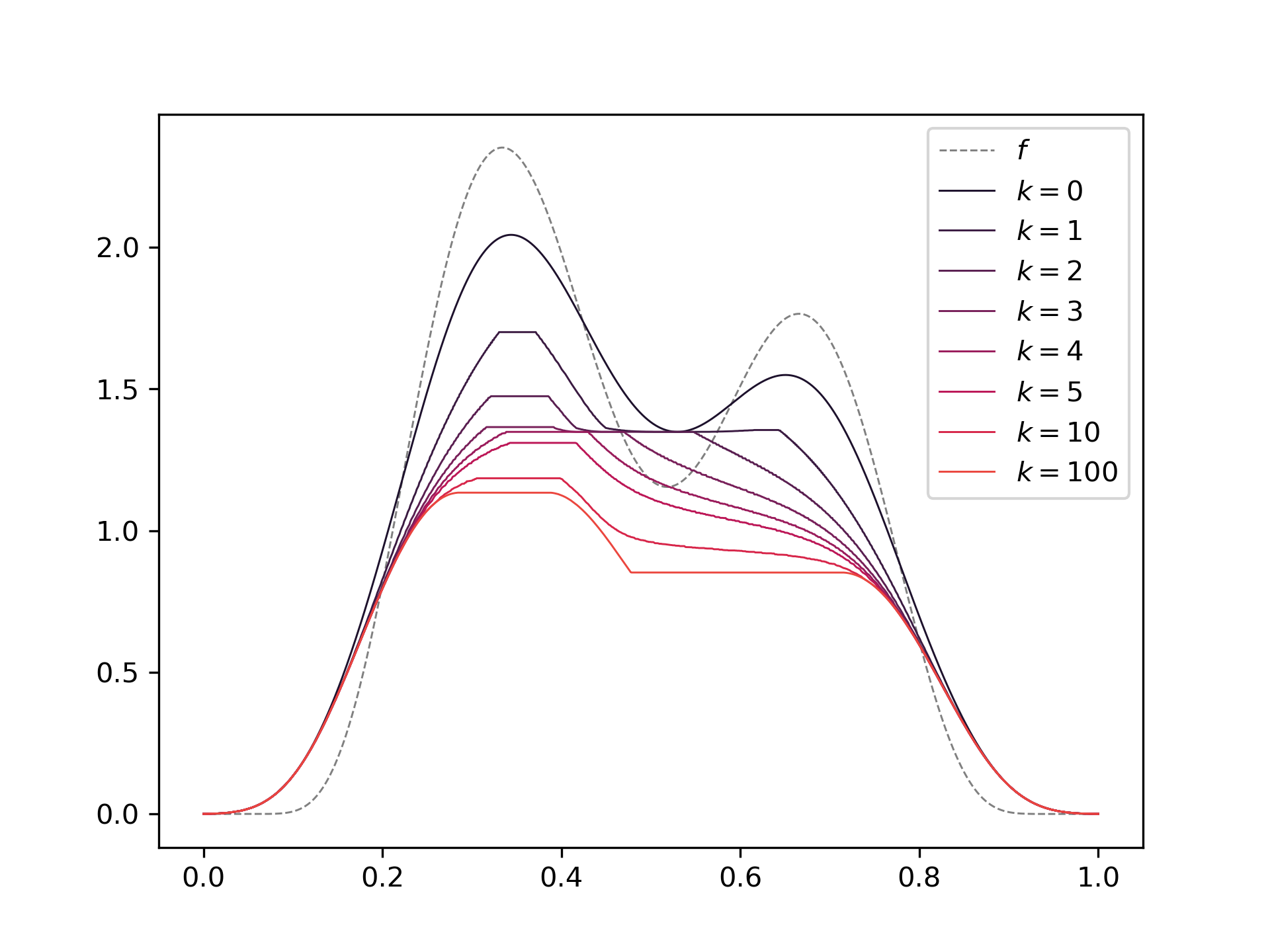}
\caption{The successive iterations of $\H_r^k f_r$ (solid) for a given density $f$ (dashed), for $k$ ranging from $0$ to $100$ with $r=0.1$. The hundredth iteration is very close to its limit $\core_r(x,f)$.}
\label{fig:hrfr}
\end{figure}

\begin{proof}
Since for all $t \geq 0$ and $x \in \bbR^d$,
$$
\frac1{\omega r^d} \int_{\ball(x,r)} \ind_{f_r(z) \geq t} f(z) \d z \leq f_r(x)
,
$$
so that $\H_r f_r \leq f_r$. Using monotonicity of the operator $\H_r$ (\lemref{monotonous}) we find that $(\H_r^k f_r)_{k \in \bbN}$ is a non-increasing sequence of functions, bounded from above by $f_r$ and from below by $0$. 
In particular, it converges towards a function $\core_r(\cdot,f)$ pointwise. Since $|f_r(x)| \leq  \sup_{\ball(x,r)} |f|$ and that the latter goes to $0$ when $x$ goes to $\infty$ (since $f$ is integrable and uniformly continuous over $\bbR^d$), we can focus on establishing the uniform convergence of $\H_r^k f_r$ on a ball 
$\ball(0,R)$ for an arbitrary large radius $R$. Having done so, the sequence $\H_r^k f_r$ is equicontinuous (from \lemref{monotonous}), and the Arzel\`a--Ascoli theorem insures that the convergence towards $\core_r(\cdot,f)$ is uniform over $\ball(0,R)$.
\end{proof}

By analogy with \eqref{eq:coreness-subgraphs}, we may also seek a variational characterization of $\core_r(x,f)$ in terms of subsets of $\bbR^d$, which are the natural continuous counterparts of subgraphs. 
This formulation, besides offering additional geometrical insights, will help with proving convergence from discrete to continuous $r$-coreness (see the proof of \thmref{fixedr}).

\begin{lem} \label{lem:cr-variational}
Let $\Omega(x)$ be the class of measurable sets $S \subset \bbR^d$ that contain $x$.
Then for $r>0$, the continuum $r$-coreness admits the following expression
\begin{align} \label{eq:crvar}
\core_r(x,f)
=
\sup\left\{ t ~\middle|~
\exists S \in \Omega(x)
\text{ such that }
\inf_{y \in S}
\frac1{\omega r^d} \int_{\ball(y, r) \cap S} f(z) \d z \ge t
\right\}
.
\end{align}
\end{lem}

\begin{proof}
Let us write $F(x)$ for the supremum on the right-hand side, and show that $\core_r(\cdot,f) = F$ by considering their super-level sets. Let $t \geq 0$, and $S = \{F \geq t\}$. 
For all $y\in \bbR^d$, we define
\[
g(y) := \frac{1}{\omega r^d}\int_{S \cap \ball(y,r)} f(z)\d z,
\] 
which, by definition of $S$, satisfies $g(y) \geq t$ for all $y \in S$. In particular, we get that for all $y \in S$,
\[
 \frac{1}{\omega r^d}\int_{\ball(y,r)} \ind_{g(z) \geq t}f(z)\d z \geq \frac{1}{\omega r^d}\int_{\ball(y,r)} \ind_{z \in S}f(z)\d z = g(y) \geq t,
\]
so that $\H_r g(y) \geq t$. 
By induction on $k\geq 1$, we find that $\H_r^k g(y) \geq t$ for all $y \in S$, and letting $k \to \infty$, that $\core_r(y,f) \geq t$ for all $y \in S$, so that $S \subset \{\core_r(\cdot,f) \geq t\}$.

For the converse inclusion, notice that since the operator $\H_r$ is $1$-Lipschitz (\lemref{monotonous}) and that $\H^k_r f_r$ converges uniformly towards $\core_r(\cdot,f)$ (\prpref{cr}), we have $\H_r \core_r(\cdot,f) = \core_r(\cdot,f)$. 
Therefore, if $y \in \{\core_r(\cdot,f) \geq t\}$, meaning $\core_r(y,f) \ge t$, by definition of $\H_r$, we get
\[
\frac{1}{\omega r^d}\int_{\ball(y,r)} \ind_{\core_r(z,f)\geq t} f(z) \d z \geq t
\]
yielding, by maximality of $S$, that $\{\core_r(\cdot,f) \geq t\} \subset S$, ending the proof.
\end{proof}

By definition, the continuum $r$-coreness $\core_r(\cdot,f)$ behaves roughly like $\H_r^k f_r$ for $k$ large enough, as shown in \figref{hrfr}.
The variational formulation of \lemref{cr-variational} also highlights the fact that $\core_r(\cdot,f)$ depends on $f$ globally, as it depends on values it takes in the entire space, at least in principle. 
That is, perturbing $f$ very far away from $x$ may change $\core_r(x,f)$ drastically. In \figref{hrfr}, this phenomenon translates into the wider and wider plateaus that $\H_r^k(\cdot,f)$ exhibits as $k$ grows, which eventually approaches $\core_r(\cdot,f)$.

We are now in position to prove the convergence of the renormalized discrete coreness towards the $r$-continuum coreness, for a bandwidth parameter $r>0$ being fixed.

\begin{thm}\label{thm:fixedr}
If $r > 0$ is fixed, then almost surely,
\beq
\frac1N \core_r(x, \cX_n) \xrightarrow[n\to\infty]{} \core_r(x, f)
\text{
uniformly in $x \in \bbR^d$.
}
\eeq
\end{thm}

\begin{proof}
Let $k \geq 1$. By the decreasingness of the iterations of the H-index $\H^k_r(x,\cX_n)$ and their convergence towards $\core_r(x,\cX_n)$ \cite[Thm 1]{lu2016h}, we have that $\core_r(x,\cX_n) \leq \H^k_r(x,\cX_n)$. Taking $n$ to $\infty$ and using \thmref{Hfixedr}, we find that almost surely,
$$
\limsup_{n\to\infty} \frac1N \core_r(x,\cX_n)  \leq \H^k_r f_r(x)
$$
uniformly in $x$, so that letting $k \to \infty$ and using \prpref{cr}, we have 
\[
\limsup_{n\to\infty} \frac1N \core_r(x,\cX_n) \leq \core_r(x,f).
\]

For the converse inequality, we will use the variational formulation of $\core_r(x,f)$ given by \lemref{cr-variational}.
Let $t < \core_r(x,f)$ and $S \subset \bbR^d$ be such that $x \in S$ and
$$
\frac{1}{\omega r^d}\int_{\ball(y,r) \cap S} f(z) \d z \geq t
\quad
\forall y\in S.
$$
Let $H$ denote the subgraph of $\cG_r(x,\cX_n)$ with vertices in $S$, and  $\deg_H$ the degree of the vertices in this subgraph. We have, for all vertex $s$ in $S$,
$$
\deg_H(s) = n \times P_n(\ball(s,r) \cap S) - 1 \geq N \times( P(\ball(s,r) \cap S) - \eta) - 1 \geq N \times( t - \eta) - 1,
$$
where
$$
\eta := \sup_{A \in \cS_r} \frac{1}{\omega r^d}|P_n(A) - P(A)|, \quad \text{with } \cS_r := \{S \cap \ball(y,r)~|~y \in \bbR^d\},
$$
so that $\core_r(x,\cX_n) \geq  N(t-\eta)-1$. The class $\cS_r$ satisfies the assumptions of \lemref{eta}, and applying that lemma with $r>0$ fixed yields that, almost surely,
$$
\liminf_{n \to \infty} \frac1N\core_r(x,\cX_n) \geq t
$$
uniformly in $x \in \bbR^d$. Letting $t \nearrow \core_r(x,f)$ establishes
$$
\liminf_{n \to \infty} \frac1N\core_r(x,\cX_n) \ge \core_r(x,f),
$$
which concludes the proof.
\end{proof}

\subsection{Continuum coreness: $r\to 0$}
Seeking to complete the construction above to include asymptotic regimes where $r \to 0$, we first opt for a purely functional approach. That is, taking the limit of the continuum $r$-coreness as $r$ goes to zero.

\begin{prp} \label{prp:c0} 
$\core_r(x,f)$ converges uniformly in $x \in \bbR^d$ as $r\to 0$. Its limit, denoted by $\core_0(x,f)$, is called the \emph{continuum coreness} at $x$. 
\end{prp}

The proof of this result relies on an intermediary notion of coreness at scale $\alpha > 0$. Intuitively, the class of subsets appearing in the variational formulation of $\core_r(\cdot,f)$ is too vast to conduct a classical VC-analysis on it. 
Instead, we shall restrict ourselves to subclasses of smoother subsets. The smoothness notion we use is indexed by a thickening parameter $\alpha > 0$. It gives rise to a new family of coreness $\core^\alpha$, closely tied with the continuous one $\core_r$ at small scales.
Given $K \subset \bbR^d$ and $y \in \bbR^d$, we write $\d(y,K) := \inf_{z \in K} \|y-z\|$ for the distance from $y$ to $K$.
We let $\cB_\alpha := \{K^{\alpha}~|~K \subset \bbR^d\}$, where $K^\alpha := \{y \in \bbR^d~|~\d(y,K) \leq \alpha\}$ and define
$$
\core^\alpha(x,f) 
:= \sup\big\{t \geq 0~\big|~\exists S \in \cB_\alpha~\text{with}~ x \in S,~ S \subset \{f \geq t\}~\text{and}~\partial S \subset \{f \geq 2t\}\big \}.
$$
Since $(K^\alpha)^\beta = K^{\alpha+\beta}$ for all $\alpha,\beta \geq 0$, the class $\cB_\alpha$ is increasing as $\alpha \to 0^+$, so is $\core^\alpha(x,f)$, and since the latter in bounded from above by $\|f\|_\infty$, it converges to a finite limit. The following lemma asserts that this limit actually coincides with the limit of $\core_r(x,f)$ as $r \to 0^+$.
\begin{lem} \label{lem:lemc0}
We have $\lim_{r \to 0} \core_r(x,f) = \lim_{\alpha \to 0} \core^\alpha(x,f)$.
\end{lem}

This result thus asserts the existence of $\core_0(x,f)$ pointwise, as used in the proof of \prpref{c0}. To show \lemref{lemc0}, we first need the following volume estimate.
\begin{lem} \label{lem:geo} For all $r \in (0,\alpha/2]$, $x \in \bbR^d$ and $y \in \ball(x,\alpha)$, we have
$$
|\ball(y,r) \cap \ball(x,\alpha) | \geq \frac12 \omega r^d \(1 - \frac{Cr}{\alpha}\),
$$
where $C$ is a positive constant depending on $d$ only.
\end{lem}
\begin{proof} The quantity $|\ball(y,r) \cap \ball(x,\alpha)|$ is a decreasing function of $\|y-x\|$, so we can only consider the case where $\|x - y\| = \alpha$. Let now
$$
\rho = r\(1 - \frac{r}{2\alpha}\)~~~\text{and}~~~x_0 = x + (\alpha - r + \rho)(y-x). 
$$
Easy computations show that the half ball 
$$
\ball^+ = \ball(x_0,\rho) \bigcap \{z \in \bbR^d~|~\inner{z-x_0}{x-x_0} \geq 0\}
$$
is a subset of $\ball(y,r) \cap \ball(x,\alpha)$ so that
\begin{align*}
|\ball(y,r) \cap \ball(x,\alpha)| 
&\geq |\ball^+| 
= 
\frac12 \omega r^d \(1-r/2\alpha\)^{d} 
\geq  
\frac12 \omega r^d (1-C r/\alpha),
\end{align*}
with $C = d/2$.
\end{proof}

\begin{proof}[Proof of \lemref{lemc0}] 
Let $0< r \leq \alpha$ and let $t = \core^\alpha(x,f)$. 
Let $K \subset \bbR^d$ be such that $K^\alpha \subset \{f \geq t-\ve\}$ and $\partial K^\alpha \subset \{f \geq 2t-2\ve\}$ for some arbitrarily small $\ve > 0$. 
For all $y \in K^\alpha$ at distance at least $r$ from $\partial K^\alpha$, we have $\ball(y,r) \subset K^\alpha$, so that
$$
\frac{1}{\omega r^d} \int \ind_{z \in K^\alpha} \ind_{z \in \ball(y,r)} f(z) \d z 
= 
\frac{1}{\omega r^d}  \int_{\ball(y,r)} f(z) \d z
\geq 
t - \ve - \omega_f(r)
,
$$
where we recall that $\omega_f$ denotes the modulus of continuity of $f$.
Otherwise if $\d(y,\partial K^\alpha) \leq r$, we have for any $v \in \ball(y,r)$ that $f(v) \geq 2t-2\ve -\omega_f(2r)$. We then have, thanks to \lemref{geo}, 
$$
|\ball(y,r) \cap K^\alpha | \geq |\ball(y,r) \cap \ball(z_0,\alpha) | \geq \omega r^d (1/2- O(r/\alpha)),
$$
where $z_0 \in K$ is such that $y \in \ball(z_0,\alpha)$.
We hence deduce that
$$
\frac{1}{\omega r^d} \int \ind_{K^\alpha} \ind_{\ball(y,r)} f \geq t - \ve - O(r/\alpha) - \omega_f(2r),
$$
so that $\core_r(x,f) \geq t - \ve -O(r/\alpha) -\omega_f(2r)$. Taking $r \to 0$ and $\ve \to 0$, we obtain $\liminf_r \core_r(x,f) \geq \core^\alpha (x,f)$, for any $\alpha > 0$.

Conversely, let $S$ be a set containing $x$ such that 
$$
\forall y \in S,~\frac{1}{\omega r^d} \int \ind_{z \in S} \ind_{z \in \ball(y,r)} f(z) \d z \geq t 
.
$$ 
In particular, we have for any $y \in S$, $f(y) \geq t - \omega_f(r)$, so that for any $y \in S^\alpha$, we have $f(y) \geq t - \omega_f(r) - \omega_f(\alpha)$. Let now take $y \in \partial S^\alpha$, and let $z_0 \in S$ be a point at distance at most $\alpha$ from $y$. We have
\begin{align*}
f(y) 
\geq 
f(z_0) - \omega_f(\alpha) 
&\geq \frac{1}{|S\cap \ball(z_0,r)|}  \int \ind_{z \in S} \ind_{z \in \ball(z_0,r)} f(z) \d z   - \omega_f(\alpha) - \omega_f(r)
\\
&\geq 
\frac{\omega r^d}{|S\cap \ball(z_0,r)|} t - \omega_f(\alpha) - \omega_f(r).
\end{align*}
But now, \lemref{geo} again yields
$$|S \cap \ball(z_0,r)| \leq |\ball(z_0,r) \setminus \ball(y,\alpha)| = \omega r^d - |\ball(z_0,r) \cap \ball(y,\alpha)| \leq \omega r^d \(1/2 + O(r/\alpha)\),$$ 
which gives
\begin{align*}
\core^\alpha(x,f) \geq t - O(r/\alpha) - \omega_f(r) - \omega_f(\alpha)
\end{align*}
and hence $\core^\alpha(x,f) + \omega_f(\alpha) \geq \limsup_r \core_r(x,f)$. 
We thus proved that
$$
\core^\alpha(x,f) \leq \liminf_r \core_r(x,f) \leq \limsup_r \core_r(x,f) \leq \core^\alpha(x,f) + \omega_f(\alpha),~~~\forall \alpha > 0,
$$
which allows to conclude.
\end{proof}

\begin{proof}[Proof of \prpref{c0}]
From \lemref{lemc0}, we get that $\core_r(\cdot,f)$ converges pointwise towards a limit $\core_0(\cdot,f)$. 
Since $\core_r(x,f) \leq f_r(x) \leq f(x) + \omega_f(r)$, and since $f \to 0$ at $\infty$ (because $f$ is integrable and is uniformly continuous), we can focus on the uniform convergence of $\core_r(\cdot,f)$ on a ball $\ball(0,R)$ for some arbitrarily large $R > 0$. But now, the uniform convergence on $\ball(0,R)$ is only a consequence of the Arzel\`a--Ascoli theorem and the equicontinuity of $\core_r(\cdot,f)$ (\remref{cr-equicontinuous}).
\end{proof}

As was shown to be the case for $\core_r(\cdot,f)$ in \lemref{cr-variational}, we also give a geometric variational formulation of $\core_0(\cdot,f)$, which is illustrated in \figref{c0-sketch}.

\begin{lem} \label{lem:c0-variational}
Let $\Sigma(x)$ be the class of open sets $S \subset \bbR^d$ with smooth boundaries that contain $x$. 
(That is, $\partial S$ is a disjoint union of smooth $(d-1)$-dimensional submanifolds of $\bbR^d$.)
Then the continuum coreness admits the following expression
$$
\core_0(x,f) = \sup\Big\{t\geq 0~|~\exists S \in \Sigma(x) \text{ such that } S \subset \{f \geq t\} \text{ and } \partial S \subset \{f \geq 2t\}\Big\}.
$$
\end{lem}

\begin{figure}
\centering
\includegraphics[width = 0.7\textwidth]{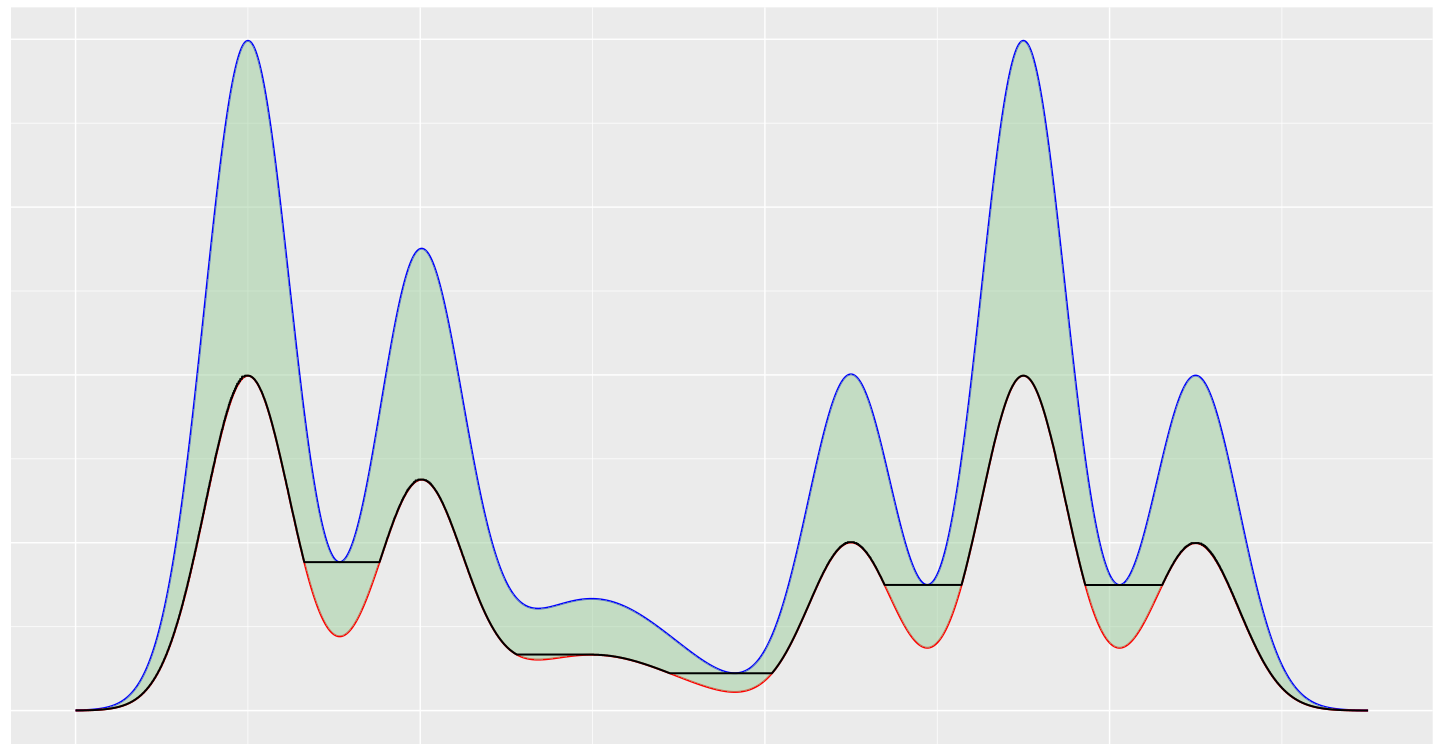}
\caption{
An illustration of $f$ (blue), $f/2$ (red) and $\core_0(\cdot,f)$ (black) for a mixture of $6$ Gaussians in dimension $d=1$.
In the zones where $\core_0(\cdot,f)$ does not coincide with $f/2$, it exhibits plateaus over intervals $[x_{\min},x_{\max}]$. For $x \in (x_{\min},x_{\max})$, the supremum of \lemref{c0-variational} is attained for $S = (x_{\min},x_{\max})$.
Otherwise, this supremum is asymptotically attained for $S = \{x\}$.
}
\label{fig:c0-sketch}
\end{figure}
\begin{proof}
Write $\core^*$ for the supremum of the right hand side. We want to show that $\core^* = \core_0(x,f)$. 
For this, take $t > 0$ such that there exists $S$ containing $x$, with smooth boundary, and such that $S \subset \{f \geq t\}$ and $\partial S \subset \{f \geq 2t\}$.  Then, for any $\alpha > 0$, $S^\alpha$ satisfies
$$
\forall y \in S^\alpha,~f(y) \geq t - \omega_f(\alpha)~~~\text{and}~~\forall y \in \partial S^\alpha,~f(y) \geq 2t - \omega_f(\alpha)
.
$$
As a result, $\core^\alpha(x,f) \geq t-\omega_f(\alpha)$ and thus, letting $\alpha \to 0$, we have $\core_0(x,f) \geq t$, and thus $\core_0(x,f) \geq \core^*$. 

Conversely, denote $t = \core_0(x,f)$ and let $\ve > 0$ and $\alpha > 0$ such that $\core_\alpha(x,f) \geq t - \ve$. There exists $K \subset \bbR^d$ containing $x$ such that $K^\alpha$ satisfies $K^\alpha \subset \{f \geq t-2\ve\}$ and $\partial K^\alpha \subset \{f \geq 2t-4\ve\}$. For $\delta > 0$, let us define
\begin{align*}
\Psi_\delta(y) := \frac{1}{\delta^d} \int_{\bbR^d}  \kappa\(\frac{y-v}{\delta}\) \ind_{K^{\alpha+\delta}}(v)\d v
,
\end{align*}
where $\kappa$ is a smooth positive normalized kernel supported in $\ball(0,1)$. The function $\Psi_\delta : \bbR^d \to \bbR$ is a smooth function with values in $[0,1]$, with $\Psi_\delta = 1$ on $K^\alpha$ and $\Psi_\delta = 0$ outside of $K^{\alpha+2\delta}$. Using Sard's lemma, we can find a regular value of $\Psi_\delta$ in $[1/4,3/4]$, say $\lambda$. The set $S = \{\Psi_\delta > \lambda\}$ is then an open set of $\bbR^d$ with smooth boundary $\partial S = \{\Psi_\delta = \lambda\}$, which contains $K$, so in particular, it contains $x$. Furthermore, any point of $S$ (resp. $\partial S$) is at distance at most $2\delta$ from $K^\alpha$ (resp. $\partial K^\alpha$). We thus have
$$
\forall y \in S,~f(y) \geq t - 2\ve - \omega_f(2\delta)~~~\text{and},~~\forall y \in \partial S^\alpha,~f(y) \geq 2t - 4\ve - \omega_f(2\delta)
,
$$
so that $\core^* \geq t- 2\ve - \omega_f(2\delta)$. Letting $\ve,\delta \to 0$, we find that $\core^* \geq \core_0(x,f)$, ending the proof.
\end{proof}

The above formulation clearly establishes that $\core_0(x,f) \leq f(x)$.
On the other hand, taking for $S$ a ball centered around $x$ with an arbitrary small radius, we find that $\core_0(x,f) \geq f(x)/2$.
The equality actually occurs whenever the homology of the super-level sets of $f$ is simple enough, as shown in \prpref{homology}. 
In particular, this is the case when the super-level sets are contractible sets (such as star-shaped ones), or the union of contractible sets.

\begin{prp} \label{prp:homology} If the complement of all super-level sets of $f$ are path-connected, then $\core_0(x,f) = f(x)/2$ for all $x \in \bbR^d$. This is the case, for example, if $f$ is a mixture of symmetric unimodal densities with disjoint supports.
\end{prp}

\begin{proof}
From the formulation of \lemref{c0-variational} applied with $S$ ranging within open balls centered at $x$ and radius $\delta \to 0$, we see that we always have $\core_0(x,f)\geq f(x)/2$.

Conversely, if $t < \core_0(x,f)$, there exists a smooth set $S \subset \{f \geq t\}$ with $\partial S \subset \{f \geq 2t\}$ that contains $x$. 
Assume for a moment that $S \setminus \{f \geq 2t\}$ is non-empty, and take a point $y$ in it.
Since by assumption $\bbR^d \setminus \{f \geq 2t\}$ is path-connected, there exists a continuous path from $y$ to any point $z \in \bbR^d \setminus S$  that stays in it. Such a path necessarily crosses $\partial S \subset \{f \geq 2t\}$, which is absurd. We hence conclude that $S \subset \{f \geq 2t\}$, so that $f(x) \geq 2t$, and taking $t$ to $\core_0(f,x)$, we find that $\core_0(f,x) \leq f(x)/2$, which concludes the proof.
\end{proof}

Hence, for densities $f$ with simple landscapes, the continuum coreness reduces to half the density.
Otherwise, generically, $\core_0(\cdot,f)$ provides us with a new functional of $f$ that lies between $f/2$ and $f$ (see \figref{c0-sketch}).
As is the case for $\core_r(\cdot,f)$, the continuum coreness $\core_0(\cdot,f)$ depends on the values $f$ on the entire space, at least in principle. This is apparent in the variational formulation of \lemref{cr-variational} and is clearly illustrated by the plateau areas of \figref{c0-sketch}.

We finally address the large-sample limit of $\core_r(x,\cX_n)$ as $r = r_n \to 0$, which does coincide with the continuum coreness $\core_0(x,f)$.

\begin{thm} \label{thm:main}
If $r = r_n$ is such that $r \to 0$ and $n r^d \gg \log n$, then almost surely, 
\beq
\frac1N \core_r(x, \cX_n) \xrightarrow[n\to\infty]{} \core_0(x, f)
~~\text{uniformly in $x \in \bbR^d$.}
\eeq
\end{thm}

The remaining results are directed towards the proof of \thmref{main}, which follows directly from \lemref{lower} and \lemref{upper}. The usual decomposition in term of variance and bias that we used for instance in the proof of \thmref{H} does not work here, because the deviation term would be indexed by a class of subsets that is too rich (and which would not satisfy the assumptions of \lemref{eta}). Instead, we take advantage of the alternative definition of the coreness through $\core^\alpha$ introduced in the beginning of this \secref{core}.

\begin{lem} \label{lem:lower} If $r = r_n$ is such that $r \to 0$ and $n r^d \gg \log n$, then almost surely, 
$$
\limsup_{n \to \infty} \frac{1}{N}\core_{r}(x,\cX_n)\leq \core_0(x,f)~~\text{uniformly in $x \in \bbR^d$.}
$$
\end{lem}
\begin{proof} 
For short, write $c_n = \core_{r}(x,\cX_n)$, and $S_n$ for the vertices of a subgraph of $\cG_{r}(x,\cX_n)$ containing $x$ with minimal degree $c_n$. Let $\alpha > 0$ and consider $S^\alpha_n \in \cB_\alpha$. For any $y \in S^\alpha_n$, there exists $s \in S_n$ such that $\|s - y\| \leq \alpha$. We deduce that
\begin{align*} 
f(y) \geq f(s) - \omega_f(\alpha) \geq \frac{1}{\omega r^d} P(\ball(s,r)) - \omega_f(r) - \omega_f(\alpha) \geq \frac{c_n}{N} - \eta - \omega_f(r) - \omega_f(\alpha).
\end{align*} 
where we denoted by
\begin{align} \label{etalower}
\begin{split}
\eta &= \sup_{S\in \cS_r} \frac1{\omega r^d} |P_n(S) - P(S)|, ~~~
\\ \text{with} ~~~ \cS_r &= \{\ball(y,r) \cap \ball(z, \alpha)~|~y,z \in \bbR^d \}~\bigcup ~\{ \ball(y,r) \setminus \ball(z, \alpha)~|~y,z \in \bbR^d \}.
\end{split} 
\end{align} 
The sets $\cS_r$ satisfy the assumptions of \lemref{eta}, so that $\eta$ goes to $0$ almost surely as $n \to \infty$.
Now, let $y \in \partial S_n^\alpha$, and take $s \in S_n$ among its nearest neighbors in $S_n$. 
This neighbor $s$ is at distance exactly $\alpha$ from $y$, so that $|S_n \cap \{\ball(s,r) \setminus \ball(y,\alpha)\}| = c_n$. But on the other hand, we have
\begin{align*} 
\frac{1}{\omega r^d} \int \ind_{\ball(s,r) \setminus \ball(y,\alpha)}(z) f(z)\d z &\leq \frac{f(s) + \omega_f(r)}{\omega r^d} \int \ind_{\ball(s,r) \setminus \ball(y,\alpha)}(z) \d z 
\\&\leq \(f(s) + \omega_f(r)\)\(1/2 + O(r/\alpha)\),
\end{align*} 
so that
\begin{align*} 
f(y)
\geq 
f(s) - \omega_f(\alpha) 
&\geq 
(2 - O(r/\alpha)) \frac{1}{\omega r^d} P(\ball(s,r) \setminus \ball(y,\alpha)) - \omega_f(r) - \omega_f(\alpha) \\
&\geq (2-O(r/\alpha)) \(\frac{c_n}{N} - \eta\)- \omega_f(r)- \omega_f(\alpha) \\
&\geq 2 \frac{c_n}{N} - O(r/\alpha) - 2\eta - \omega_f(r)- \omega_f(\alpha).
\end{align*} 
Putting the two estimates of $f$ over $\partial S^\alpha$ and $S^\alpha$ together, we have shown that
$$
\core^\alpha(x,f) \geq \frac{c_n}{N} - O(r/\alpha) - \eta - \omega_f(r)- \omega_f(\alpha)
,
$$
so that, using \lemref{eta}, we have almost surely,
$$
\limsup_{n \to \infty} \frac{c_n}{N} \leq \core^\alpha(x,f) + \omega_f(\alpha)~~\text{uniformly in $x \in \bbR^d$.}
$$
Letting $\alpha \to 0$ then concludes the proof.
\end{proof}

\begin{lem}  \label{lem:upper}
If $r = r_n$ is such that $r \to 0$ and $n r^d \gg \log n$, then almost surely, 
$$
\liminf_{n \to \infty} \frac{1}{N} \core_r(x,\cX_n) \geq \core_0(x,f)~~\text{uniformly in $x \in \bbR^d$.}
$$
\end{lem}

\begin{proof} Let  $\alpha > 0$ and $\ve > 0$. Denoting $t = \core^\alpha(x,f)$, there is $S \in \cB_\alpha$ with $x \in S$ such that 
$$
\forall y \in S,~f(y) \geq t - \ve~~~\text{and}~~~\forall y \in \partial S,~f(y) \geq 2t - 2\ve.
$$
Let $H$ be the subgraph of $\cG_r(x,\cX_n)$ with vertices in $S$, and let $\deg_H(s)$ be the degree of a vertex $s \in S$ in $H$. If $s$ is at distance more than $r$ from $\partial S$, then, using again $\eta$ introduced in the proof of \lemref{lower} at \eqref{etalower},
$$
\deg_H(s) = n\times P_n(\ball(s,r)) - 1\geq N (f(s) - \omega_f(r) -  \eta) - 1 \geq N(t - \ve - \omega_f(r) -  \eta) - 1. 
$$
Now if $s$ is at distance less that $r$ than $\partial S$, we can take $y \in S$ such that $s \in \ball(y,\alpha) \subset S$. The volume of $\ball(s,r) \cap \ball(y,\alpha)$ is then at least $\omega r^d(1/2 - O(r/\alpha))$ according to \lemref{geo}. We thus have, 
\begin{align*} 
\deg_H(s) &= n\times P_n(S \cap \ball(s,r)) - 1
\\
&\geq  
n\times P_n(\ball(y,\alpha) \cap \ball(s,r))  - 1 
\\
&\geq N\(\frac{1}{\omega r^d} P(\ball(y,\alpha) \cap \ball(s,r)) - \eta\)-1 \\
&\geq N\(\(\frac12 - O(r/\alpha)\)(f(s) - \omega_f(r)) - \eta\)-1 \\
&\geq N\(t - \ve - O(r/\alpha) - \omega_f(r) - \eta\)-1
\end{align*} 
where we used the fact that $f(s) \geq 2t - 2\ve - \omega_f(r)$ because $s$ is $r$-close to $\partial S$. We thus have shown here that
$$
\frac{\core_r(x,\cX_n)}{N} \geq t - \ve - O(r/\alpha) - \omega_f(r) - \eta-1/N
.
$$
Now letting $n \to \infty$ yields, almost surely,
$$
\liminf_{n \to \infty} \frac{\core_r(x,\cX_n)}{N} \geq t - \ve = \core^\alpha(x,f) - \ve~~~\text{uniformly in $x \in \bbR^d$.}
$$
and letting $\alpha, \ve \to  0$ yields the result.
\end{proof}

\section{Numerical Simulations}
\label{sec:numeric}

We performed some small-scale proof-of-concept computer experiments to probe into the convergences established earlier in the paper, as well as other questions of potential interest not addressed in this paper. 

\subsection{Illustrative Examples}
\label{sec:plots}

In the regime where $r = r_n \to 0$ and $nr^d \gg \log(n)$, Theorems~\ref{thm:degree_rfixed},~\ref{thm:H} and~\ref{thm:main} show that only $f(x)$ and $\core_0(x,f)$ can be obtained as limits of H-index iterates $\H_r^k(x,\cX_n)$, when $k \in \{0,1,\ldots,\infty\}$ is fixed.
Figures~\ref{fig:h-index-illustration-1d} and~\ref{fig:h-index-illustration-2d} both illustrate, for $d=1$ and $d=2$ respectively, the following convergence behavior:
\begin{itemize}
\item
$\frac1{N} \deg_r(x ,\cX_n) \xrightarrow[n\to\infty]{} f(x)$ (see \thmref{degree_rfixed});
\item
$\H_r^k(x,\cX_n)\xrightarrow[k\to\infty]{} \core_r(x,\cX_n)$ (see \eqref{eq:coreness-hindex});
\item
$\frac1{N} \core_r(x,\cX_n) \xrightarrow[n\to\infty]{} \core_0(x,f)$ (see \thmref{main}).
\end{itemize}
The density functions have been chosen to exhibit non-trivial super-level sets, so that $\core_0(\cdot,f) \neq f/2$ (see \prpref{homology}).

\begin{figure}[!htbp]
\centering
\begin{subfigure}{1\textwidth}
	\centering
	\includegraphics[width = 0.9\linewidth]{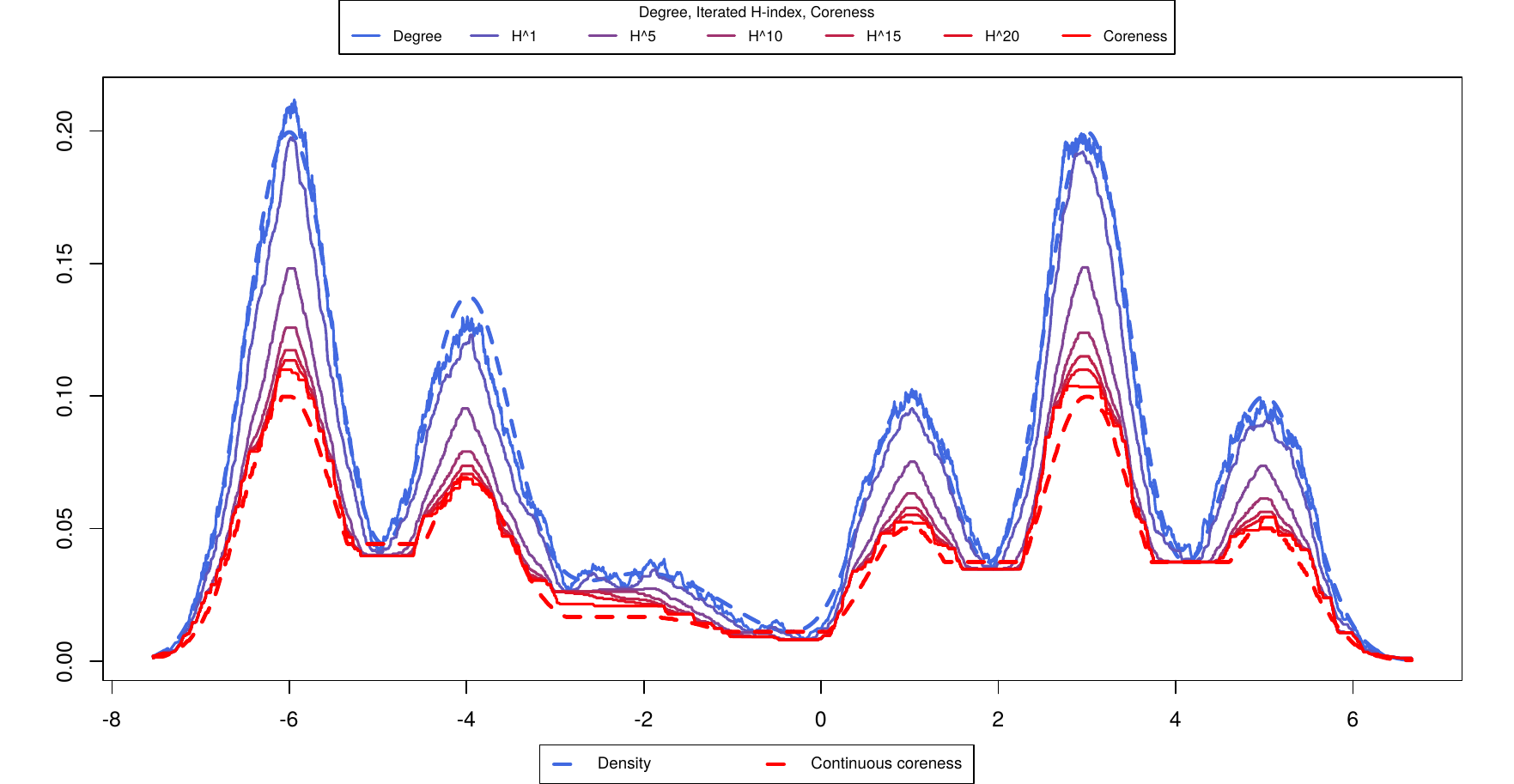}
\caption{
The mixture $f$ of six Gaussians in dimension $d=1$ from \figref{c0-sketch} sampled $n = 10 000$ times. On the discrete side (solid), are displayed the degree ($k=0$), iterated H-indices for $k \in \{1,5,10,15,20\}$, and coreness $(k=\infty$). On the continuous side (dashed), the density $f$ and the continuum coreness $\core_0(\cdot,f)$ are plotted.
Here, $r \approx 0.13$ was picked proportional to the optimal kernel bandwidth $r_\opt \asymp n^{-1/(d+2)} = n^{-1/3}$.
}
\label{fig:h-index-illustration-1d}
\end{subfigure}

\begin{subfigure}{1\textwidth}
	\centering
\includegraphics[width = 0.9\linewidth]{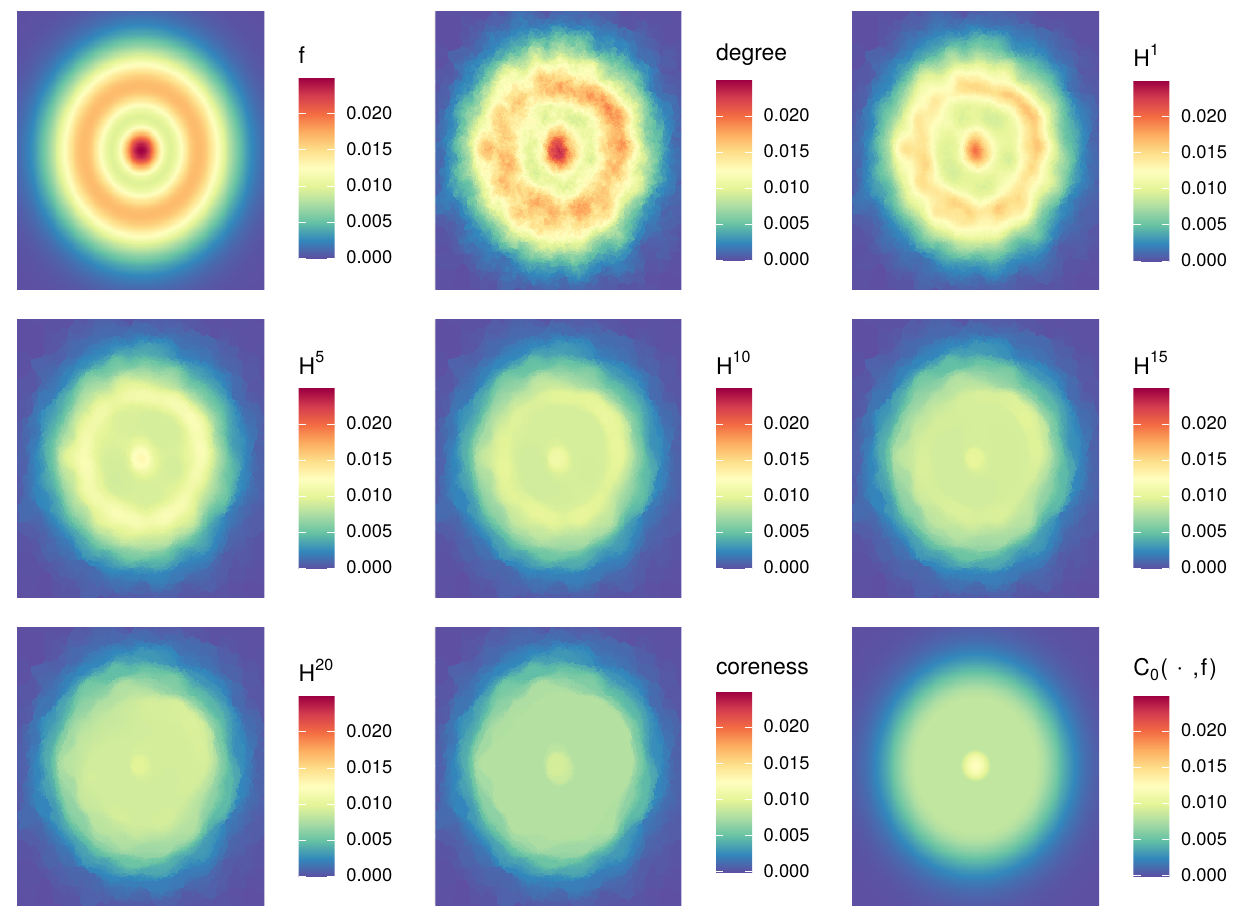}
\caption{A plot similar to \figref{h-index-illustration-1d} for $d = 2$. The generating density function $f$ exhibits a crater-like shape enclosing a peak, yielding a continuum coreness $\core_0(\cdot,f)$ that plateaus, and in particular differs from $f/2$ within the crater area. Here, $n = 20000$, $k \in \{0,1,5,10,15,20,\infty\}$ and $r \asymp r_\opt \asymp n^{-1/(d+2)} = n^{-1/4}$.
}
\label{fig:h-index-illustration-2d}
\end{subfigure}
\caption{Illustrative examples in dimension $d=1$ and in dimension $d=2$.}
\label{fig:h-index-illustration}
\end{figure}

\subsection{Convergence Rates}
\label{sec:rates}

Intending to survey limiting properties of the degree, the H-index and the coreness, the above work does not provide convergence rates.
We now discuss them numerically in the regime where $r \to 0$.

A close look at the proofs indicates that only bias terms of order $O(r \vee \omega_f(r))$ appear in the convergences of Theorems~\ref{thm:H} and~\ref{thm:main}.
For the degree, the stochastic term is known to be of order \smash{$O\bigl(1/\sqrt{nr^d}\bigr)$}.
If $f$ is Lipschitz (i.e., $\omega_f(r) = O(r)$), the bandwidth $r_\opt$ that achieves the best minimax possible convergence rate in \thmref{degree_rfixed} is $r_\opt = O(n^{-1/(d+2)})$, yielding a pointwise  error $|N^{-1}\deg_r(x, \cX_n) - f(x)| = O(r_\opt) = O(n^{-1/(d+2)})$.
Naturally, larger values $r \geq r_\opt$ make the bias term lead, and smaller values $r \leq r_\opt$ make the stochastic term lead.
Although it remains unclear how bias terms behave for H-indices and the coreness, simulations indicate a similar bias-variance tradeoff depending on $n$ and $r$. 
Indeed, the sup-norms $\|N^{-1}\deg_r(\cdot, \cX_n) - f\|_\infty$ and $\|N^{-1}\core_r(\cdot, \cX_n) - \core_0(\cdot,f)\|_\infty$ appear to be linearly correlated (see \figref{scatterplot-convergence-rate}).
\begin{figure}[!htbp]
\centering
\begin{subfigure}{0.49\textwidth}
	\centering
	\includegraphics[width = 1\linewidth]{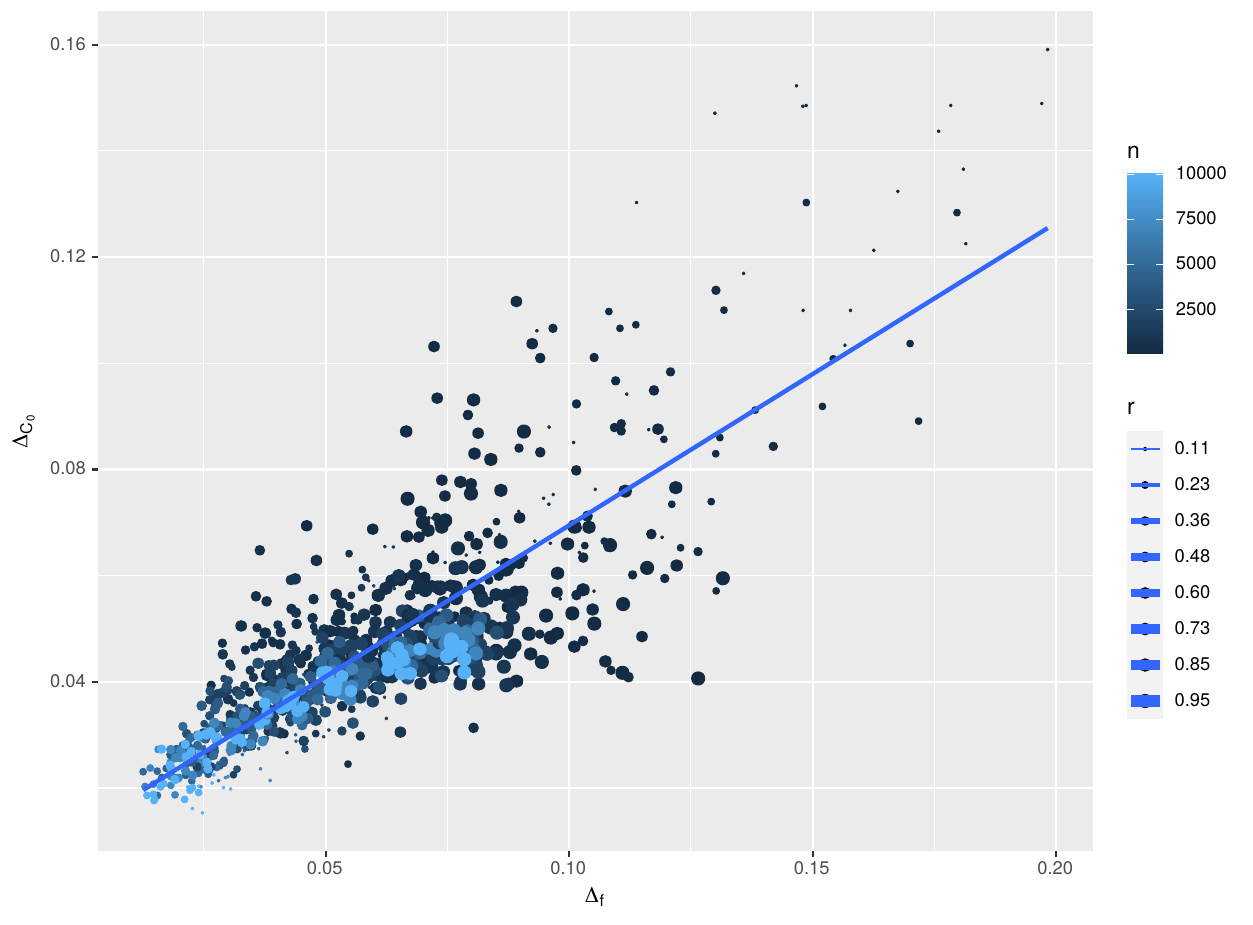}
	\caption{Distribution of \figref{h-index-illustration-1d} ($d=1$).}
	\label{fig:subfig_scatterplot1d}
\end{subfigure}
\begin{subfigure}{0.49\textwidth}
	\centering
	\includegraphics[width = 1\linewidth]{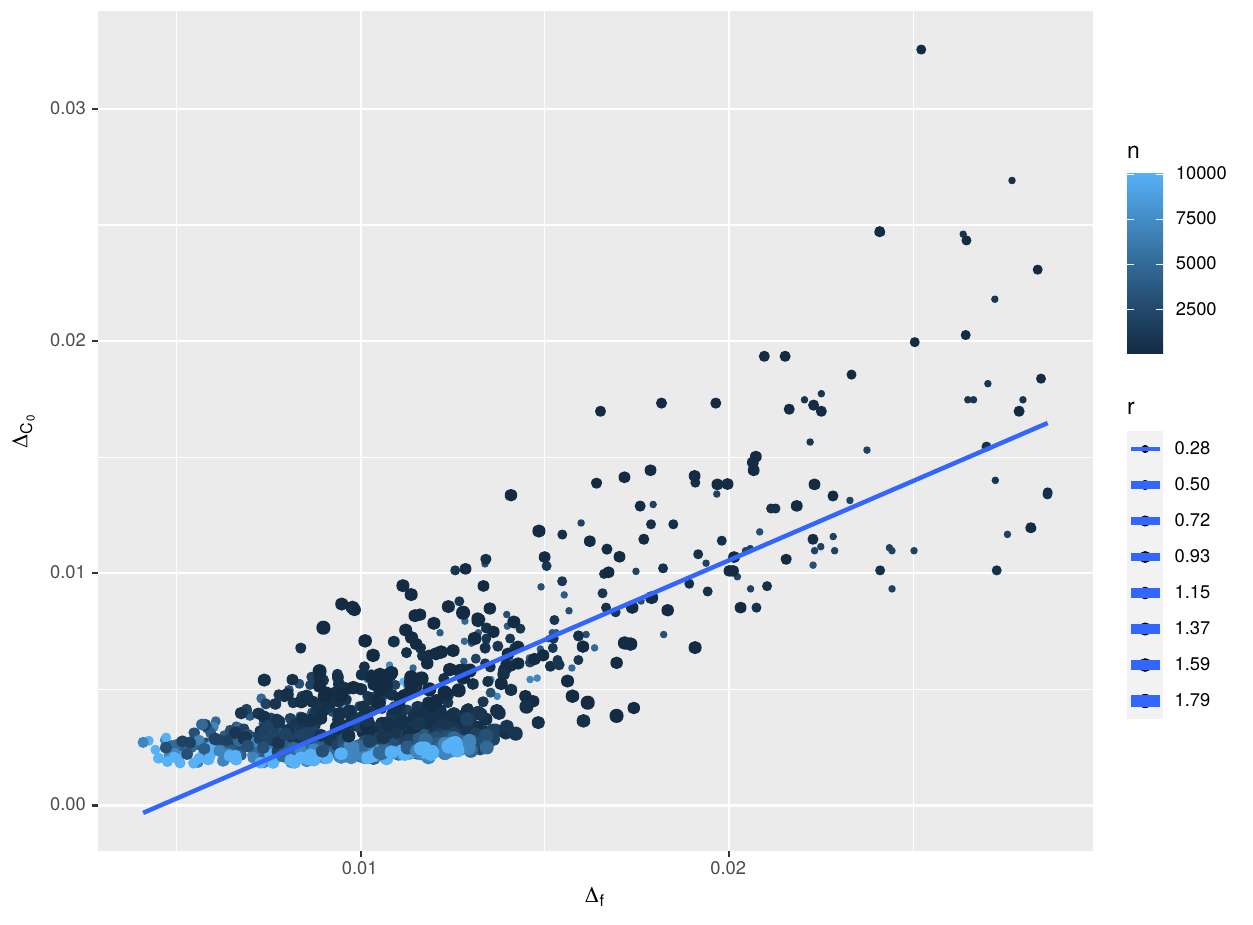}
	\caption{Distribution of \figref{h-index-illustration-2d} ($d=2$).}
	\label{fig:subfig_scatterplot2d}
\end{subfigure}
\caption{Scatterplot of values $\bigl(\|N^{-1}\deg_r(\cdot, \cX_n) - f\|_\infty,\|N^{-1}\core_r(\cdot, \cX_n) - \core_0(\cdot,f)\|_\infty\bigr)$ with data generated according to (\subref{fig:subfig_scatterplot1d}) the Gaussian mixture distribution depicted in \figref{h-index-illustration-1d}, and~(\subref{fig:subfig_scatterplot2d}) the crater-like density of \figref{h-index-illustration-2d}.
Sample size values $n$ take $9$ different values in $[100,10000]$, while connection radii $r$ take $8$ different values within the interval $[0.1,0.97]$ for (\subref{fig:subfig_scatterplot1d}) and $[0.27,1.80]$ for~(\subref{fig:subfig_scatterplot2d}).
For each pair $(n,r)$, simulations are repeated $10$ to $20$ times, depending on the value of $n$.
}
\label{fig:scatterplot-convergence-rate}
\end{figure}
As a result, with a choice $r \asymp r_\opt = O(n^{-1/(d+2)})$, we anticipate that with high probability,
\begin{align}\label{eq:conjecture-rate}
\notag
|N^{-1}\core_r(x, \cX_n) - \core_0(x)| 
&= 
O(|N^{-1}\deg_r(x, \cX_n) - f(x)|)
\hspace{1em}
\\
&= O(n^{-1/(d+2)})
,
\tag{Rate Conjecture}
\end{align}
Furthermore, \figref{scatterplot-convergence-rate} suggests that the slope relating the random variables $\|N^{-1}\deg_r(\cdot, \cX_n) - f\|_\infty$ and $\|N^{-1}\core_r(\cdot, \cX_n) - \core_0(\cdot,f)\|_\infty$ is of constant order, in fact between $1/2$ and $1$, which suggests very moderate constants hidden in the $O(|N^{-1}\deg_r(x, \cX_n) - f(x)|)$.

\subsection{Iterations of the H-Index}
\label{sec:k_max}
Seen as the limit \eqref{eq:coreness-hindex} of H-index iterations, the coreness $\core_r(x,\cX_n) = \H_r^\infty(x,\cX_n)$ raises computational questions.
One of them resides in determining whether it is reasonable to compute it naively, by iterating the H-index over the graph until stationarity at all the vertices.

More generally, given a graph $G=(V,E)$ and a vertex $v\in V$ of $G$ , and similarly as what we did in \secref{h-index} for random geometric graphs, we can study the H-index $\H_G(v)$, its iterations $\H^k_G(v)$ for $k \in \bbN$, and the coreness $\core_G(v)$. 
The \emph{max-iteration $k^\infty(G)$} of the H-index of $G$ is then defined as the minimal number of iterations for which the iterated H-index $\H_G^k$ coincides with the coreness $\core_G$. That is,
$$
k^\infty(G) := \min\big\{k \in \bbN~|~ \forall v \in V,~\core_G(v) = \H^k_G(v)\big\}.
$$
Known bounds for $k^\infty(G)$ are of the form
$$
k^\infty(G) \leq 1 + \sum_{v \in V} |\deg_G(v) - \core_G(v)|~~~~\text{and}~~~~ k^\infty(G) \leq |V|,
$$
and can be found in  \cite[Thm 4 \& Thm 5]{Mon13}. For random geometric graphs, this yields probabilistic bounds of order $O(n^2 r^d)$ and $O(n)$ respectively, with one or the other prevailing depending on whether we are in a sub-critical or super-critical regime.

However, for the random geometric graphs $\cG(x,\cX_n)$, numerical simulations suggest that an even stronger bound of order $k^\infty(\cG_r(x,\cX_n)) = O(nr^{d-1})$ may hold with high probability (see \figref{k_max}).
Indeed, in the regime where $r = r_n$ is large enough that $\cG_r(x,\cX_n)$ is connected, this latter quantity appears to coincide with its diameter --- which is of order $O(1/r)$ --- multiplied by its maximal degree --- which is of order $O(nr^d)$.

Coming back to the general deterministic case, this observation leads us to conjecture that
\begin{align}\label{eq:conjecture-k_max}
\tag{Max-Iter. Conjecture}
\hspace{4em}
k^\infty(G)
&\leq
\max_{\substack{H \subset G \\ \text{connected}}}
\diam(H)
\
\times
\max_{v \in V} \deg_G(v)
,
\hspace{6em}
\end{align}
where $\diam(H)$ is the diameter of $H$ seen a combinatorial graph (with edge weight $1$).
This conjecture, clearly satisfied in simulations (see \figref{k_max}), would shed some light --- if correct --- on the dependency of the H-index iteration process with respect to the graph's geometry.
\begin{figure}[!htbp]
\centering
\begin{subfigure}{0.49\textwidth}
	\centering
	\includegraphics[width = 1\linewidth]{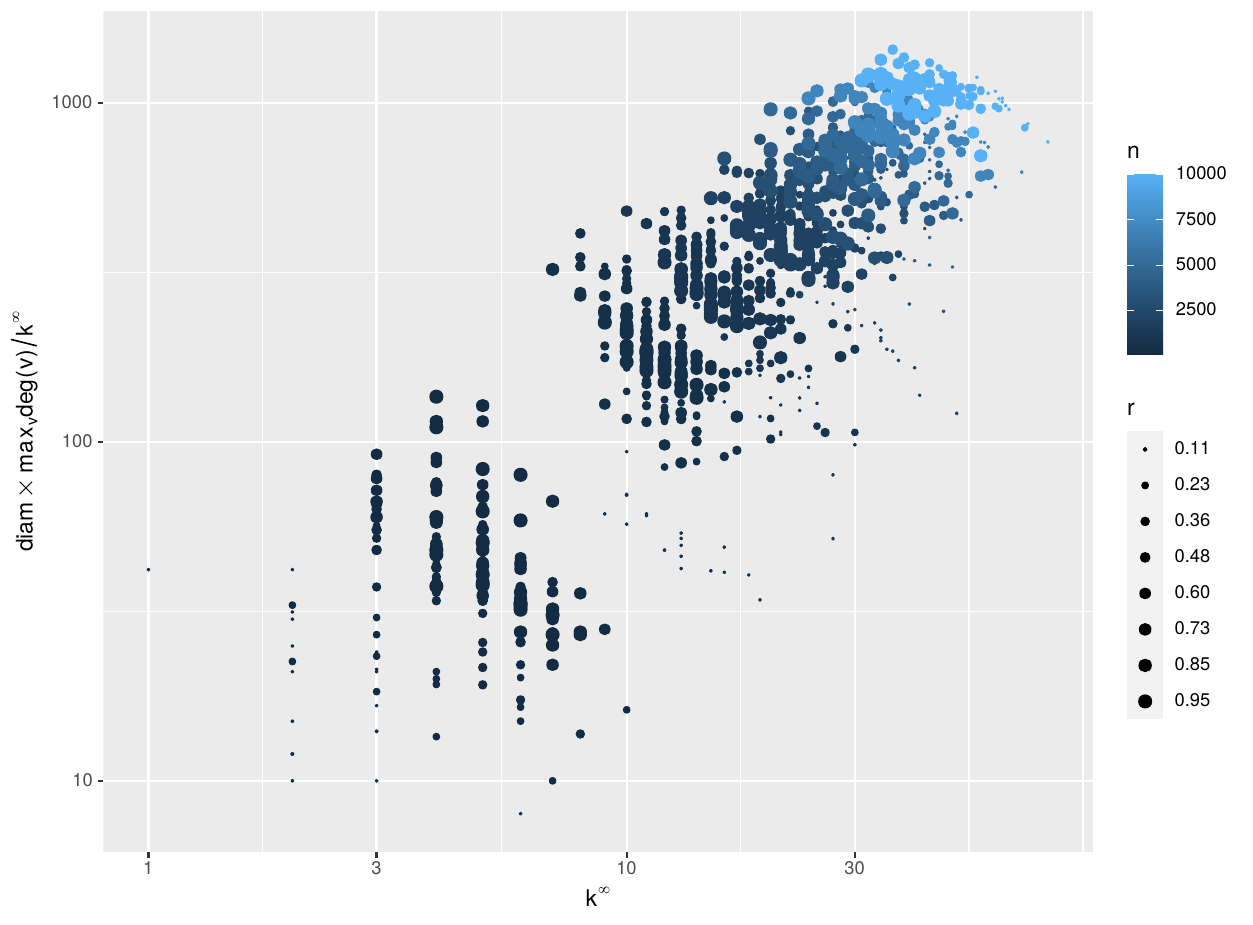}
	\caption{Distribution of \figref{h-index-illustration-1d} ($d=1$).}
	\label{fig:subfig_k_max_1d}
\end{subfigure}
\begin{subfigure}{0.49\textwidth}
	\centering
	\includegraphics[width = 1\linewidth]{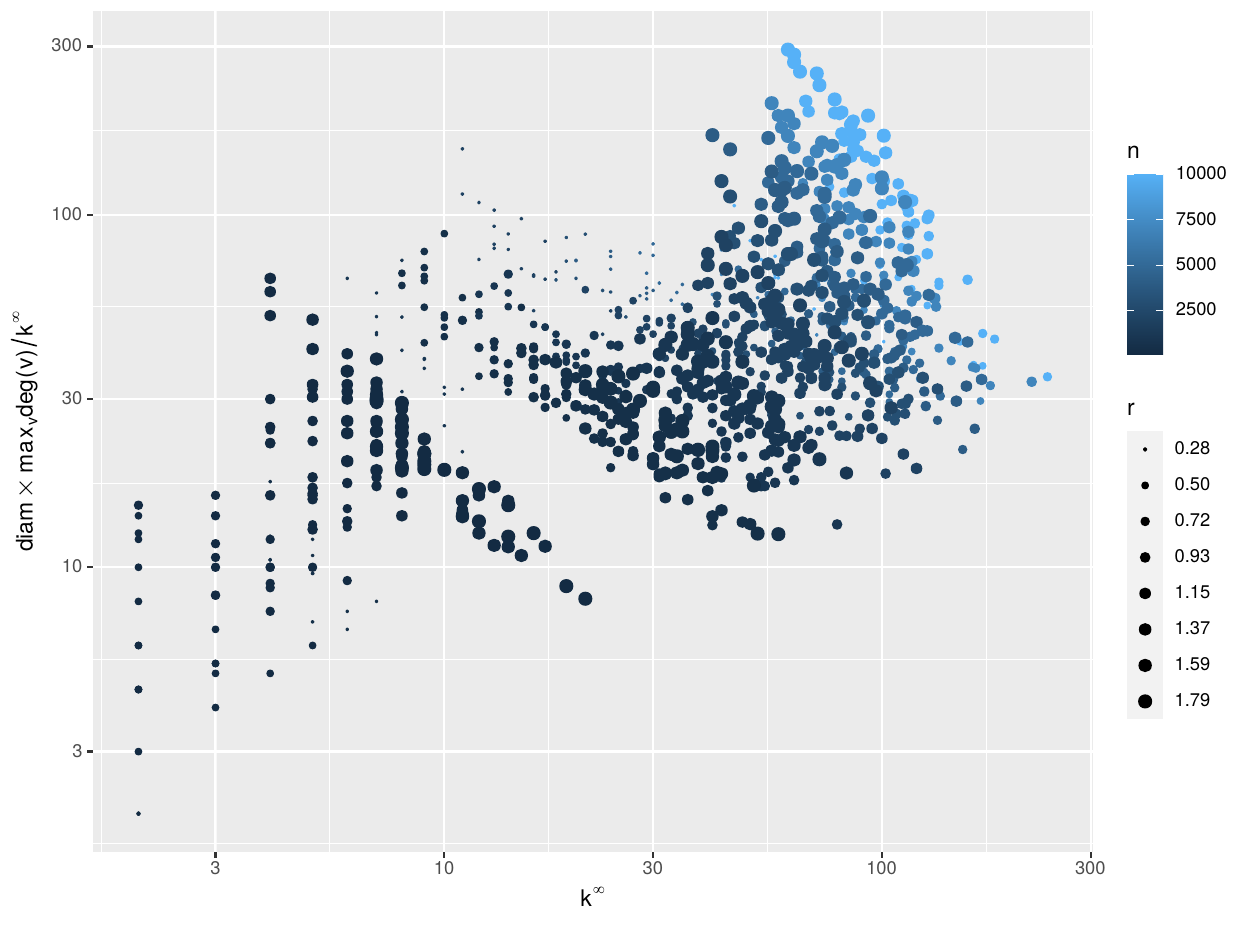}
	\caption{Distribution of \figref{h-index-illustration-2d} ($d=2$).}
	\label{fig:subfig_k_max_2d}
\end{subfigure}
\caption{Scatterplot of values $\bigl(k^\infty(\cG_r), \diam(\cG_r) \times \max_v \deg_{\cG_r}(v) /k^\infty(\cG_r) \bigr)$ in log-log scale, with data generated according to (\subref{fig:subfig_k_max_1d}) the Gaussian mixture distribution of \figref{h-index-illustration-1d}, and~(\subref{fig:subfig_k_max_2d}) the crater-like density of \figref{h-index-illustration-2d}.
Values all appear to satisfy $k^\infty(\cG_r) \leq \diam(\cG_r) \times \max_v \deg_{\cG_r}(v)$ widely (i.e., points with ordinate at least $1$ in these plots), even for small values of $r$ and $n$.
}
\label{fig:k_max}
\end{figure}

\section{Concluding Remarks and Discussions}
\label{sec:discussion}

\subsection{Generalizations}

Given a random neighborhood graph, we obtained limits for the centrality measures given by the degree, H-index, and coreness. 
In the same spirit, it would be interesting to study the limiting objects associated to other notions of centrality, such as the closeness centrality of~\cite{freeman1978centrality}, the betweenness of~\citep{freeman1977set}, and other `spectral' notions~\citep{katz1953new, bonacich1972factoring, page1999pagerank, kleinberg1999hubs}.
Similarly, we only focused on a $r$-ball neighborhood graph construction, but there are other graphs that could play that role, such as nearest-neighbor (possibly oriented) graphs or Delaunay triangulations, possibly yielding new limiting objects.

\subsection{Applications}
In the context of multivariate analysis, a notion of {\em depth} is meant to provide an ordering of the space~\citep[Property~1]{serfling2000general}.  While in dimension  one there is a natural order (the one inherited by the usual order on the real line), in higher dimensions this is lacking, and impedes the definition of such foundational objects as a median or other quantiles, for example.
While the focus in multivariate analysis is on point clouds, in graph and network analysis the concern is on relationships between nodes in a graph, with their importance in the graph being measured through various notions of centralities.

Thus, on the one hand, notions of depth have been introduced in the context of point clouds, while on the other hand, notions of centrality have been proposed in the context of graphs and networks, and these two lines of work seem to have evolved completely separately, with no cross-pollination whatsoever, at least to our knowledge.\footnote{The only place where we found a hint of that is in the discussion of \cite{aloupis2006geometric}, who mentions a couple of ``graph-based approach[es]'' which seem to have been developed for the context of point clouds, although one of them --- the method of \cite{toussaint1979some} based on pruning the minimum spanning tree --- applies to graphs as well.}
This lack of interaction may appear surprising, particularly in view of the important role that neighborhood graphs have played in multivariate analysis, for example, in areas like manifold learning \citep{tenenbaum2000global, weinberger05nonlinear, belkin2003laplacian}, topological data analysis \citep{wasserman2016topological, chazal2011geometric}, and clustering~\citep{ng2002spectral, arias2011clustering, maier2009optimal, brito1997connectivity}.

At this stage, we do not know whether graph-based notions of depth offer promise in regards to methodology. The notions that we study in the present article either lead to the likelihood, which does not appear to be particularly popular, or to somewhat peculiar functions which --- as the expert in data depth will notice --- do not define notions of depths which are affine invariant in general, while affine invariance is one of the desired properties of a depth function~\citep[Property~1]{serfling2000general}.
We can, however, imagine a situation where a notion of depth derived from a notion of graph centrality may work well.
For instance, graph-based notions of depth may adapt to the intrinsic geometry of a low-dimensional sampled manifold, a context --- common in modern multivariate analysis --- in which the standard notions of depth are all doomed.

\subsection*{Acknowledgments}
This work was partially supported by the US National Science Foundation (DMS 1513465).
The work of EAC was partially supported by the US National Science Foundation (DMS 1916071).
The work of CB was supported by the Deutsche Foschungsgemeinschaft (German Research Foundation) on the French-German PRCI ANR ASCAI CA 1488/4-1 “Aktive und Batch-Segmentierung,
Clustering und Seriation: Grundlagen der KI”.
We are grateful to Charles Arnal for insightful comments.

\bibliographystyle{chicago}
\bibliography{ref}

\end{document}